\documentclass[10pt,twoside]{siamart1116}
\usepackage[english]{babel}
\usepackage{graphicx,epstopdf,epsfig}
\usepackage{amsfonts,epsfig,fancyhdr,graphics, hyperref,amsmath,amssymb}

\def\cvd{~\vbox{\hrule\hbox{%
     \vrule height1.3ex\hskip0.8ex\vrule}\hrule } }

\newcommand{\HE}{Name of Handling Editor}
\newcommand{\DoS}{Month/Day/Year}
\newcommand{\DoA}{Month/Day/Year}
\newcommand{\CA}{Fernando De Ter\'{a}n}

\newcommand{\Names}{Fernando De Ter\'an, Carla Hernando, Javier P\'erez}
\newcommand{\Title}{Structured strong $\boldsymbol{\ell}$-ifications for structured matrix polynomials in the monomial basis}

\renewcommand{\theequation}{\thesection.\arabic{equation}}
\usepackage{enumerate}
\usepackage{multirow}
\usepackage{mathdots}
\usepackage[all]{xy}

\newcommand{\la}{{\lambda}}
\newcommand{\FF}{{\mathbb{F}}}
\newcommand{\cM}{{\mathbb{M}}}

\newtheorem{Teo}{Theorem}[section]
\newtheorem{Def}[Teo]{Definition}

\newtheorem{Rem}[Teo]{Remark}

\newtheorem{Lem}[Teo]{Lemma}
\newtheorem{Prop}[Teo]{Proposition}
\newtheorem{Exa}[Teo]{Example}

\newenvironment{smallarray}[1]
{\null\,\vcenter\bgroup\scriptsize
	\arraycolsep=.13885em
	\hbox\bgroup$\array{@{}#1@{}}}
{\endarray$\egroup\egroup\,\null}

\begin{document}

%\bibliographystyle{plain}

%  Leave these commented lines here
%\input{ELAheader-template.tex}
% ELA insert correct page number
\setcounter{page}{1}

\thispagestyle{empty}

%Insert the title of the paper
\title{
\Title\thanks{Received by the editors on \DoS. Accepted for publication on \DoA. Handling Editor: \HE. Corresponding Author: \CA. \newline
	This work has been partially supported by the {\em Ministerio de Econom\'ia  Competitividad} of Spain through grants MTM2017-90682-REDT and MTM2015-65798-P.}
}

\author{
	Fernando de Ter\'an \thanks{Departamento de Matem\'aticas, Universidad Carlos III de Madrid, Avenida de la Universidad 30, 28911 Legan\'es, Spain (fteran@math.uc3m.es, cahernan@math.uc3m.es).}
% Remember to put \and between any two authors
\and
Carla Hernando \footnotemark[2]
\and 
Javier P\'erez \thanks{Department of Mathematical Sciences, University of Montana, Missoula, MT 59812, USA (javier.perez-alvaro@mso.umt.edu).}
}

\markboth{\Names}{\Title}

\maketitle

\begin{abstract}
In the framework of Polynomial Eigenvalue Problems, most of the matrix polynomials arising in applications are structured polynomials (namely (skew-)symmetric, (skew-)Hermitian, (anti-)palindromic, or alternating). The standard way to solve Polynomial Eigenvalue Problems is by means of linearizations. The most frequently used linearizations belong to general constructions, valid for all matrix polynomials of a fixed degree, known as {\em companion linearizations}. It is well known, however, that is not possible to construct companion linearizations that preserve any of the previous structures for matrix polynomials of even degree. This motivates the search for more general companion forms, in particular {\em companion $\ell$-ifications}. In this paper, we present, for the first time, a family of (generalized) companion $\ell$-ifications that preserve any of these structures, for matrix polynomials of degree $k=(2d+1)\ell$. We also show how to construct sparse $\ell$-ifications within this family. Finally, we prove that there are no structured companion quadratifications for quartic matrix polynomials.
\end{abstract}


% Sample article for the Electronic Journal of Linear Algebra

%%%%%%%%%%%%%%%%%%%%%%%%%%%%%%%%%%%%%%%%%%%%%%%%%%%%%%%%%%%%%

\section{Introduction} 

Matrix polynomials
\begin{equation}\label{mpoly}
	P(\la) = \sum_{j=0}^k \la^j P_j,\ {\rm with\ } P_0,\hdots,P_k \in \FF^{n \times n}
\end{equation}
(where $\FF$ is an arbitrary field) arise frequently associated with {\em Polynomial Eigenvalue Problems} (PEPs) (see, for instance, \cite{BHMS13,mehrmann2004nonlinear,TiMe01}). The standard way to solve PEPs is by means of {\em (strong) linearizations} (see Definition \ref{StrongLifications}), which are matrix pencils (namely, matrix polynomials of grade $1$, see Section \ref{preli_sec}), that allow to recover the relevant information of the PEP (the so-called {\em eigenstructure}). Two particular interesting features of linearizations from the point of view of applications are being ``easily constructible" from the coefficients, $P_0,\hdots,P_k$, of the polynomial, and valid for all matrix polynomials of a given grade. This is the case, for instance, of the classical {\em Frobenius companion pencils}
\begin{equation}\label{frobenius}
	F_1(\la):=\, \la \left[ \begin{array}{c@{\mskip8mu}c@{\mskip8mu}c@{\mskip8mu}c} P_k&&&\\[1pt]&I_n&&\\[-3pt]&&\ddots&\\&&&I_n\end{array}\right] \,+\,
	\left[ \begin{array}{c@{\mskip8mu}c@{\mskip8mu}c@{\mskip8mu}c}
		P_{k-1} & P_{k-2} & \cdots & P_0 \\[1pt]
		-I_n & 0 & \cdots & 0 \\[-3pt]
		& \ddots & \ddots & \vdots\\
		0 & & -I_n & 0
	\end{array}\right],\quad {\rm and} \quad F_2(\la)=F_1(\la)^{\mathfrak B}, %\,\in\, \FF[\la]^{k \times k}.
\end{equation}
where $(\cdot)^{\mathfrak B}$ means block transposition). These pencils are particular cases of {\em companion linearizations} for grade-$k$ matrix polynomials \eqref{mpoly}. These are block-partitioned symbolic constructions that depend on the coefficients of the polynomial and which are strong linearizations for all matrix polynomials of a fixed grade. One of the main advantages of these linearizations is that they are easily constructible (you just need to place the coefficients of the polynomial in some specific positions, together with other entries, typically some identities, though there can be other additional entries involved) and that they are valid for all matrix polynomials. In the recent years, several families of companion linearizations have been introduced in the literature \cite{antoniou2004new,bueno2014eigenvectors,bueno2011recovery,de2010fiedler,dopico2018block,vologiannidis2011permuted}. 

One relevant feature of matrix polynomials that usually arise in applications is the {\em structure}. Here, the term ``structure" refers to certain symmetries in the entries of the coefficients, $P_0,\hdots,P_k$, of the matrix polynomial. In particular, the most frequently studied structures in the context of PEPs are the (skew-) symmetric, the (skew-)Hermitian, the(anti-)palindromic, and the alternating structures \cite{AhMe13,BHMS13,BoKaMeSha14,BoMe06,HiMM04,mackey2006structured,mackey2009numerical,MeXu15} (see Definition \ref{Structures-Mobius} for the formal definition of all these structures). The symmetries on the coefficients that define all these structures imply certain symmetries in the spectral structure of the matrix polynomial (for instance, in the {\em spectrum}). If these symmetries in the coefficients of the polynomial are not preserved in the linearization, then, as a consequence of the rounding errors that would affect the corresponding numerical method (which is applied to the linearization), the symmetries in the spectrum would be lost, providing meaningless results. Therefore, it is important to look for linearizations that preserve these symmetries in the coefficients of the matrix polynomial. We will refer to these linearizations as {\em structured linearizations}. The interest of the community in such constructions resulted in several families of structured linearizations (not companion) \cite{AnVo06,mackey2006structured}. Later on, structured companion linearizations were obtained for the palindromic structure \cite{bueno2012palindromic,de2011palindromic}, for the symmetric structure \cite{bueno2014structuredI,bueno2018calcolo,bueno2014structuredII,bueno2018explicit}, and for all of them \cite{dopico2019structured} (note that some of these are quite recent references). However, these structured companion linearizations are not valid for all polynomials. Actually, as it was discovered in \cite[Ths. 7.20 and 7.21]{de2014spectral}, there are no companion linearizations for structured matrix polynomials with even grade, for any of the structures mentioned so far.

The non-existence of structured companion linearizations motivates the search for other constructions that allow to recover the eigenstructure of the matrix polynomial \eqref{mpoly} through a matrix polynomial with higher grade $\ell\geq1$. A particular case of these constructions is the {\em companion $\ell$-ifications} (see Definitions \ref{CompLification} and \ref{GenCompLification}). These are block-partitioned matrix polynomials of grade $\ell$, with blocks of size $n\times n$ (like companion linearizations) that contain the coefficients, $P_0,\hdots,P_k$, of the polynomial among their blocks, and that allow to recover the eigenstructure of the matrix polynomial, as companion linearizations do (in particular, they include companion linearizations, when $\ell=1$). Companion $\ell$-ifications for general (non-structured) matrix polynomials have been considered in several recent references \cite{binirobol16,de2014spectral,de2016ellifications,dopico2019block}. In particular, a relevant family, known as {\em companion block-Kronecker $\ell$-ifications}, has been recently introduced in \cite{dopico2019block}, extending the family of block-Kronecker linearizations \cite{dopico2018block}. 

However, so far no families of structured companion $\ell$-ifications are known, apart of some particular examples provided in \cite[Ex. 5]{dopico2018block} for the symmetric structure, or the strong linearization introduced in \cite{huang2011palindromic} for the palindromic structure, which is quasi-companion, but not exactly companion (see Section \ref{SectionNoStructCompLific}).

In this paper, we introduce a family of structured $\ell$-ifications, valid for all divisors, $\ell$, of the grade $k$ of the matrix polynomial \eqref{mpoly} such that $k=(2d+1)\ell$, for some $d$, namely, $k$ is the product of $\ell$ times an odd number. We obtain a family of $\ell$-ifications for any of the structures mentioned so far (that is, the ones in Definition \ref{Structures-Mobius}). Our construction is a particular case of the block-minimal bases $\ell$-ifications introduced in \cite{dopico2018block}. More precisely, we choose a particular minimal basis that allows us to create structured constructions which lie in the family of \cite{dopico2019block}. One of the features of these families is that they allow to recover the {\em minimal indices} (which is part of the eigenstructure, and not preserved by general $\ell$-ifications) of the matrix polynomial \eqref{mpoly} from the minimal indices of any $\ell$-ification in the families. In particular, we provide explicit and very elementary formulas for the minimal indices of the matrix polynomial in terms of the minimal indices of these $\ell$-ifications. These formulas are a direct consequence of the general recovery formulas provided in \cite{dopico2019block} for the block-minimal bases $\ell$-ifications.

We are particularly interested in {\em sparse} constructions, namely those having the smallest possible number of nonzero block entries. Sparsity is a feature that raises when looking for the simplest constructions, but also has a numerical motivation, since one could devise more efficient algorithms by taking advantage of the sparsity. We first determine the smallest possible number of nonzero block entries in any $\ell$-ification in the families considered in the paper, and then we provide a procedure to construct $\ell$-ifications with exactly this number of nonzero block entries. In particular, this number is only attainable for the palindromic structure, since, as we see in Corollary \ref{CorollarySparse}, for the (skew)-symmetric, (skew-)Hermitian, and alternating structures, all $\ell$-ifications in the proposed family have a larger number of nonzero entries (this number is determined in the proof of this corollary).

After the previous considerations, a natural question arises: Is it possible to obtain similar constructions when $k=(2d)\ell$? In this paper, we deal mainly with the notion of what we call {\em generalized companion $\ell$-ification} (see Definition \ref{GenCompLification}), instead of the one of {\em companion $\ell$-ification} (or {\em companion form}), which is used, for instance, in \cite[Def. 1.1]{de2011palindromic} and \cite[Def. 5.1]{de2014spectral} (see Definition \ref{CompLification}). The first notion allows for more flexibility in the blocks. More precisely, they are polynomials in the coefficients, $P_0,\hdots,P_k$, of the matrix polynomial \eqref{mpoly} so, for instance, sums and products of these coefficients may appear in the blocks of a generalized companion $\ell$-ification. However, in the last notion, the only nonzero blocks allowed are of the form $I_n$ or $P_j$, for $j=0,\hdots,k$, up to multiplication by scalars (namely, numbers in $\FF$). This notion of companion form is consistent with the use of the term ``companion" in several references on companion matrices for scalar polynomials, like \cite{eastman2014companion,eastman2016pentadiagonal}. Despite its apparently restrictive definition, all the seminal families of companion-like linearizations (namely, the Fiedler-like families \cite{antoniou2004new,bueno2014eigenvectors,bueno2011recovery,de2010fiedler,vologiannidis2011permuted}) consist of this kind of constructions, and the block-Kronecker family \cite{dopico2018block} also contains many pencils of this form. Moreover, all the structured families that we present in this work include such constructions. Therefore, to consider companion $\ell$-ifications instead of generalized companion $\ell$-ifications makes sense. So the question at the beginning of this paragraph can be split in the following two questions: (Q1) Is it possible to obtain structured generalized companion $\ell$-ifications for matrix polynomials with grade $k=(2d)\ell$? and (Q2) Is it possible to obtain structured companion $\ell$-ifications for matrix polynomials with grade $k=(2d)\ell$? The answer to (Q1) is positive. In particular, it is possible to easily construct structured generalized companion quadratifications ($\ell=2$) for palindromic matrix polynomials of grade $k=4d$ from the quadratifications introduced in \cite{huang2011palindromic} (see the example at the beginning of Section \ref{SectionNoStructCompLific}). However, these constructions involve the product of several coefficients of the original matrix polynomial \eqref{mpoly}, so they are not just companion. In Section \ref{SectionNoStructCompLific}, we prove that there are no structured companion quadratifications for quartic matrix polynomials (that is, $k=4$), which means that the answer to (Q2) is negative when $d=1$ and $\ell=2$. Then, constructions like the ones mentioned above are not possible (at least for $k=2\ell$ and $\ell=2$) if we restrict ourselves to companion forms instead of generalized companion forms.

Finally, we want to emphasize that in this work we are just considering square matrix polynomials which are expressed in the monomial basis. First, it makes no sense to consider rectangular polynomials, because the structure imposes that the matrix polynomial must have square coefficients. However, most of the families of linearizations or $\ell$-ifications known so far are valid for rectangular matrix polynomials (see, for instance, \cite{de2012fiedler,dopico2018block,dopico2019block}). As for the polynomial basis, several authors have considered linearizations of matrix polynomials in different polynomials bases \cite{amiraslani2008linearization,ashkar2020linearizations,corless2007generalized,fassbender17,lawrence2017constructing,noferini2016fiedler,robol2017framework}. 

The paper is organized as follows. In Section \ref{preli_sec} we introduce or recall the notation and the basic notions and tools that are used in further sections of the manuscript. Section \ref{bmb-sec} presents the general construction (namely, the ``(strong) block-minimal bases matrix polynomials") that motivates our structured $\ell$-ifications, together with the fundamental result on which they rely (Theorem \ref{SLification}). A natural extension to structured matrix polynomials of grade $\ell$ is also introduced. Section \ref{KroneckerStructured} is the main section of this manuscript, where we present our constructions. The main results are Theorems \ref{NoStructLlambdaTeo} and \ref{StructLgeneral}. Then, in Subsection \ref{SectionSparse}, we analyze the sparse structured $\ell$-ifications within the new families. In Section \ref{SectionNoStructCompLific} we show that there are no structured companion quadratifications for quartic matrix polynomials (Theorem \ref{NoStructCompLific}) and, finally, in Section \ref{conclusion_sec}, we summarize the main results of the manuscript and present some lines of further research.

\section{Preliminaries}\label{preli_sec}

Throughout the paper we use the following notation. Let $\FF$ be an arbitrary field, and we denote by $\FF[\la]$ the ring of polynomials in the variable $\la$ with coefficients in $\FF$ and by $\FF(\la)$ the field of rational functions with coefficients in $\FF$. The algebraic closure of $\FF$ is denoted by $\overline{\FF}$. The set of $m \times n$ matrices with entries in $\FF[\la]$ or $\FF(\la)$ are, respectively, denoted by $\FF[\la]^{m \times n}$ or $\FF(\la)^{m \times n}$. 

We say that the nonzero matrix polynomial $P(\la)$ in \eqref{mpoly} has {\it grade} $k$. Note that some of the coefficients $P_j$, including the leading coefficient $P_k$, may be the zero matrix. The {\it degree} of $P(\la)$, denoted by $\deg(P)$, is the largest integer, $j \geq 0$, such that $P_j \neq 0$. In particular, when $k=1$, we say that $P(\la)$ is a {\it matrix pencil}. A matrix polynomial is said to be {\it unimodular} if its determinant is a nonzero constant.

We recall the notation for some matrix operations that we use in this paper. The set of $n \times n$ invertible matrices with entries in $\FF$ is denoted by $GL(n, \FF)$. Given two matrices $A$ and $B$, $A \otimes B$ denotes their Kronecker product \cite{horn1994}. Now, let $A \in \FF^{m \times n}$ be a constant matrix. Then, $A^\top$ denotes the transpose of $A$, $A^\ast$ denotes the conjugate transpose of $A$, which is used only when $\FF = \mathbb{C}$, and $A^\star$ denotes either the transpose or the conjugate transpose. Finally, when $\FF=\mathbb{C}$ again, $\overline{A}$ is a matrix whose entries are the conjugate of the entries of $A$, so $\overline{A} = A$ when $\FF = \mathbb{R}$. When these operations on constant matrices are applied on matrix polynomials, the matrix polynomial is seen as follows:
$P^{\diamondsuit}(\la) = \sum_{j=0}^k \la^j P_j^\diamondsuit$, where $\diamondsuit = \{\top,\ast,\star \}$ and $\overline{P}(\la) = \sum_{j=0}^k \la^j \overline{P_j}$, i.e., these operations are just applied on the coefficients of the matrix polynomial $P(\la)$, and not to the variable $\la$.

Another notion for square matrices that plays an important role in this work is the one of {\it coninvolotury} matrix \cite{horn1993real}, which is a matrix $A \in \FF^{n \times n}$ such that $A \cdot \overline{A} = I_n$. When $\FF = \mathbb{R}$, coninvolutory matrices are known as {\it involutory} matrices, and the condition reads $A^2 = I_n$.

In our analysis of strong $\ell$-ifications for matrix polynomials, we consider the ${\rm rev}$ operator. In what follows, the {\it $k$th reversal} of $P(\la)$ is the polynomial ${\rm rev}_kP(\la) := \sum_{j=0}^k \la^j P_{k-j}$, obtained by reversing the order of the coefficients of $P(\la)$. 

\begin{Def}{\rm\cite[Def. 3.3]{de2014spectral}}\label{StrongLifications}
{\rm
	A matrix polynomial $L(\la)$ of grade $\ell$ is said to be an {\rm $\ell$-ification} of a given $n \times n$ matrix polynomial $P(\la)$ of grade $k$ if, for some $s \ge 0$, there exist two unimodular matrix polynomials $U(\la)$ and $V(\la)$ such that 
	$$U(\la)L(\la)V(\la) = \begin{bmatrix} I_s & 0 \\ 0 & P(\la) \end{bmatrix}.$$
	If, additionally, the matrix polynomial ${\rm rev}_\ell L(\la)$ is an $\ell$-ification of ${\rm rev}_k P(\la)$, then $L(\la)$ is said to be a {\rm strong $\ell$-ification} of $P(\la)$. When $\ell=1$, (strong) $\ell$-ifications are called (strong) linearizations. When $\ell=2$, (strong) $\ell$-ifications are called (strong) quadratifications.
}
\end{Def}

The main motivation for dealing with (strong) $\ell$-ifications of a matrix polynomial is to recover all the (scalar) spectral information of the polynomial, namely: (a) the finite elementary divisors (finite eigenvalues together with their multiplicities), (b) the infinite elementary divisors (infinite eigenvalue, together with its multiplicities), (c) right minimal indices, and (d) left minimal indices. 

We refer the reader to \cite[Section 7.1]{de2014spectral} for more details on the spectral information. 

Any strong $\ell$-ification $L(\la)$ of a matrix polynomial $P(\la)$ has the same finite and infinite elementary divisors as $P(\lambda)$ \cite[Theorem 4.1]{de2014spectral} and the same number of left (resp. right) minimal indices. However, the set of left (resp. right) minimal indices of $L(\la)$ and $P(\la)$ can be different. Therefore, in order to recover all the spectral information, it is important to provide formulas to recover the left and right minimal indices of $P(\lambda)$ from those of the $\ell$-ification. In Theorems \ref{NoStructLlambdaTeo} and \ref{StructLgeneral}, which are the main new results in Sections \ref{bmb-sec}--\ref{KroneckerStructured}, we will show the relationship between the left and right minimal indices of the proposed $\ell$-ifications and those of the original polynomial.

The notion of {\it companion form (of grade $\ell$)} was introduced, for the first time, in \cite[Def. 5.1]{de2014spectral} as a uniform template of grade-$\ell$ matrix polynomials, which are constructed from the entries in the coefficients of $P(\la)$, in such a way that they are strong $\ell$-ifications for every matrix polynomial $P(\la)$. We refer the reader, again, to \cite{de2014spectral} for more information on this concept and its properties. Nonetheless, as the authors commented just after \cite[Def. 5.1]{de2014spectral}, the most common way to build such $\ell$-ifications is by means of a template obtained after block-partitioning each coefficient of the grade-$\ell$ matrix polynomial in such a way that each nonzero block is either an identity or a coefficient, $P_0,\hdots,P_k$, of the matrix polynomial (up to constants). This more restrictive definition corresponds to the notion of {\em companion form} in \cite{de2011palindromic} for matrix pencils, and its natural extension to grade-$\ell$ matrix polynomials leads to the notion of {\it companion form (of grade $\ell$)} (or {\it companion $\ell$-ification}, to abbreviate) that we will use in this paper. It is presented in Definition \ref{CompLification}.

%The following notion was introduced in \cite[Def. 5.1]{de2014spectral} under the name of {\it companion form (of grade $\ell$)}. We refer the reader, again, to \cite{de2014spectral} for more information on this concept and its properties. 

\begin{Def}\label{CompLification}
{\rm
	Given a divisor, $\ell$, of $k$, a {\it companion $\ell$-ification} for $n \times n$ matrix polynomials $P(\lambda)$ of grade $k$ is an $\frac{nk}{\ell} \times \frac{nk}{\ell}$ block partitioned matrix polynomial $L(\lambda) = \sum_{i=0}^\ell \la^i L_i$ of grade $\ell$ with $\frac{k}{\ell} \times \frac{k}{\ell}$ blocks of size $n \times n$ such that:
	\begin{enumerate}[{\rm (a)}]
		\item each nonzero block of $L_i$, for $i=0,\hdots,\ell$, is either of the form $\alpha I_n$ or $\alpha P_j$, for some $j = 0,1,\hdots,k$, and $\alpha \in \FF\backslash\{0\}$, and
		%\item each nonzero block of $L_j$, for $j=0,\hdots,\ell$, belongs to $\FF(P_0,\hdots,P_k)$, and
		\item $L(\la)$ is a strong $\ell$-ification for every $n \times n$ matrix polynomial of grade $k$.
	\end{enumerate}
}
\end{Def}

The following notion is an extension, to matrix polynomials of grade $\ell$, of the family of {\it generalized companion pencils} introduced in \cite{de2020gen}. This notion allows for more flexibility in the nonzero blocks of $L(\la)$ than Definition \ref{CompLification}.

\begin{Def}\label{GenCompLification}
{\rm 
	A {\it generalized companion $\ell$-ification} for $n \times n$ matrix polynomials $P(\lambda)$ of grade $k$ as in \eqref{mpoly} is a $\frac{nk}{\ell} \times \frac{nk}{\ell}$ block partitioned matrix polynomial $L(\lambda) = \sum_{i=0}^\ell \la^i L_i$ of grade $\ell$ with $\frac{k}{\ell} \times \frac{k}{\ell}$ blocks of size $n \times n$ such that:
	\begin{enumerate}[{\rm (a)}]
		%\item each nonzero block of $L_j$, for $j=0,\hdots,\ell$, is either $I_n$ or $A_i$ (up to scalar constants) for some $i = 0,1,\hdots,k$, and
		\item each nonzero block of $L_i$, for $i=0,\hdots,\ell$, is a polynomial in $P_0,\hdots,P_k$ (i.e., it belongs to $\FF[P_0,\hdots,P_k]$), and
		\item $L(\la)$ is a strong $\ell$-ification for every $n \times n$ matrix polynomial of grade $k$.
	\end{enumerate}
}
\end{Def}

Our goal in this paper is the construction of structure-preserving strong $\ell$-ifications for structured matrix polynomials, for any of the structures introduced in Definition \ref{Structures-Mobius}. By a {\em structured (generalized) companion $\ell$-ification} we mean a (generalized) companion $\ell$-ification with the additional property that it is structured whenever $P(\lambda)$ is. These structured $\ell$-ifications will be addressed in Section \ref{KroneckerStructured}.

All the families of strong $\ell$-ifications that we present in this paper are generalized companion $\ell$-ifications. We particularize to the notion of structured companion $\ell$-ification (not generalized) only in Section \ref{SectionNoStructCompLific}, where we prove that there are no structured companion quadratifications for structured quartic matrix polynomials.

\subsection{Dual minimal bases}

Another fundamental concept in this work is the one of {\it dual minimal bases}. The name of ``dual minimal bases" and its definition were introduced in \cite[Def. 2.10]{de2016polynomial}, but their origins go back to \cite{forney1975minimal}. This is the essential tool in the development of the block-Kronecker linearizations and $\ell$-ifications \cite{dopico2018block}--\cite{dopico2019block} and also for linearizations in other bases other than the monomial basis \cite{robol2017framework}.

%In order to define the notion of dual minimal bases, we first need to recall the following notions. The {\it normal rank} of a matrix polynomial $P(\la) \in \FF[\la]^{m \times n}$, denoted by ${\rm rank}(P(\la))$, is defined as the rank of $P(\la)$ over the field $\FF(\la)$.

By ${\rm rank}\,P(\la_0)$ we mean the rank of the constant matrix $P(\la_0)$, obtained by evaluating the matrix polynomial $P(\la)$ as in \eqref{mpoly} at $\la_0$. We say that $P(\la_0)$ has {\it full row} (resp. {\it column}) {\it rank} if ${\rm rank}\, P(\la_0) = m$ (resp. ${\rm rank}\,P(\la_0) = n$). The {\it $i$th row degree} of a matrix polynomial $P(\la)$ is the degree of the $i$th row of $P(\la)$ (namely, the largest degree of the entries of this row).

We also need the following notions, that can be found, for instance, in \cite{de2016polynomial}.

\begin{Def}
{\rm 
	Let $P(\la) \in \FF[\la]^{m \times n}$ be a matrix polynomial with row degrees $k_1,\hdots,k_m$. The {\it highest row degree coefficient matrix} of $P(\la)$, denoted by $P_h$, is the $m \times n$ constant matrix whose $i$th row is the coefficient of $\la^{k_i}$ in the $i$th row of $P(\la)$, for $i=1,\hdots,m$. The matrix polynomial $P(\la)$ is called {\it row reduced} if $P_h$ has full row rank.
}
\end{Def}

\begin{Def}\label{minbases}
{\rm 
	The matrix polynomial $P(\la) \in \FF[\la]^{m \times n}$ is a {\it minimal basis} if and only if the following conditions are satisfied:
	\begin{enumerate}[{\rm (a)}]
		\item $P(\la)$ is row reduced, and
		\item $P(\la_0)$ has full row rank for all $\la_0 \in \overline{\FF}$.
	\end{enumerate}
}
\end{Def}

\begin{Def}\label{dualminbases}
{\rm 
	Two matrix polynomials $P(\la) \in \FF[\la]^{m_1\times n}$ and $Q(\la) \in \FF[\la]^{m_2\times n}$ are called {\it dual minimal bases} if the following conditions are satisfied:
	\begin{enumerate}[{\rm (a)}]
		\item $P(\la)$ and $Q(\la)$ are both minimal bases,
		\item $m_1+m_2=n$, and
		\item $P(\la)Q^\top(\la) = 0$.
	\end{enumerate}
}
\end{Def}

The following matrix polynomials, which are dual minimal bases (see Definition \ref{dualminbases}) are known as ``block-Kronecker dual minimal bases". We will use these dual minimal bases in Section \ref{KroneckerStructured} to construct structured generalized companion $\ell$-ifications of a matrix polynomial $P(\lambda)$. The main advantage of these particular minimal bases relies on their simplicity, which allows, in particular, to easily relate the $\ell$-ification with the original polynomial $P(\lambda)$:
\begin{eqnarray}
L_d(\la)&:=&\begin{bmatrix} -1 & \la & & & \\ & -1 & \la & & \\ & & \ddots & \ddots & \\ & & & -1 & \la \end{bmatrix} \in \FF[\la]^{d \times (d+1)}, \ \ {\rm and} \label{Ld} \\
\Lambda_d^\top(\la)&:=&\begin{bmatrix} \la^d & \cdots & \la & 1 \end{bmatrix} \in \FF[\la]^{1 \times (d+1)}. \label{Lambdad}
\end{eqnarray}

Note that $L_d(\la)$ and $\Lambda_d^\top(\la)$ satisfy all the conditions in Definition \ref{dualminbases}, so they are dual minimal bases. The following lemma shows how to easily obtain other pairs of dual minimal bases extending the block-Kronecker dual minimal bases to higher degrees.

\begin{Lem}{\rm\cite[Lemma 3.6]{dopico2019block}}\label{dualminbasesell}
	Let $L_d(\la)$ and $\Lambda_d^\top(\la)$ be the matrix polynomials defined, respectively, in \eqref{Ld} and \eqref{Lambdad}. Then, for any $\ell \in \mathbb{N}$ the following statements hold.
	\begin{enumerate}[{\rm (a)}]
		\item The matrix polynomials $L_d(\la^\ell)$ and $\Lambda_d^\top(\la^\ell)$ are dual minimal bases.
		\item For any $n \in \mathbb{N}$, the matrix polynomials $L_d(\la^\ell) \otimes I_n$ and $\Lambda_d^\top(\la^\ell) \otimes I_n$ are dual minimal bases.
	\end{enumerate}
\end{Lem}
Notice that $L_d(\la^\ell)$ and $\Lambda_d^\top(\la^\ell)$ are matrix polynomials with constant row degrees $\ell$ and $\ell d$, respectively. The same holds for the matrix polynomials in Lemma \ref{dualminbasesell}-(b).

Theorem \ref{dualminbasesrev} establishes several properties of minimal bases. In particular, it is shown that there exists a direct relation between a minimal basis, whose row degrees are all equal, and its reversal. 

\begin{Teo}{\rm\cite[Th. 3.7]{dopico2018block}}\label{dualminbasesrev}
	The following statements hold.
	\begin{enumerate}[{\rm (a)}]
		\item Let $P(\la)$ be a minimal basis whose row degrees are all equal to $j$. Then ${\rm rev}_jP(\la)$ is also a minimal basis whose row degrees are all equal to $j$. 
		\item Let $P(\la)$ and $Q(\la)$ be dual minimal bases. If the row degrees of $P(\la)$ are all equal to $j$ and the row degrees of $Q(\la)$ are all equal to $i$, then ${\rm rev}_jP(\la)$ and ${\rm rev}_iQ(\la)$ are also dual minimal bases with all the row degrees of ${\rm rev}_jP(\la)$ equal to $j$ and all the row degrees of ${\rm rev}_iQ(\la)$ equal to $i$.
	\end{enumerate}	
\end{Teo}

\subsection{M\"obius transformations and $\boldsymbol{\mathbb{M}_A}$-structured matrix polynomials}\label{SectionMa}

M\"obius transformations have become a relevant tool in the theory of structured matrix polynomials. They have been used in the framework of matrix polynomials, at least, since $2006$ to relate structured matrix polynomials having different structures \cite{mackey2006structured}. More recently, in \cite{dopico2019structured} the authors showed that M\"obius transformations for matrix polynomials allow to provide a common framework for the most frequent classes of structured matrix polynomials, such as (skew-)symmetric, (skew-)Hermitian, (anti-)palindromic, and alternating polynomials. A thorough study on the influence of M\"obius transformations on relevant properties of general matrix polynomials over arbitrary fields has been carried out in \cite{mackey2015mobius}.
\begin{Def}\label{DefMaP}
{\rm
	Let $A = \left[\begin{smallarray}{cc} a & b \\ c & d \end{smallarray}\right] \in GL(2,\FF)$. The {\it M\"obius transformation} of $P(\la) := \sum_{j=0}^k \la^j P_j$ induced by $A$ is defined by
	\begin{equation*}
	\cM_A[P](\la) := \sum_{j=0}^k P_j(a\la + b)^j(c\la + d)^{k-j}.%,\ {\rm where\ } A = \left[\begin{smallarray}{cc} a & b \\ c & d \end{smallarray}\right].
	\end{equation*}
}
\end{Def}

In order to simplify the exposition, from now on we write $\cM_A[P] := \cM_A[P](\la)$. We will write the parameter $\la$ only when needed.

The following proposition includes general properties of M\"obius transformations that we will use throughout this work. They can be found, for instance, in \cite[Prop. 3.4 and Th. 3.5]{dopico2019structured}.

\begin{Prop}\label{PropersMaPs}
	For any $A,B \in GL(2,\FF)$ the following statements hold.
	\begin{enumerate}[{\rm (a)}]
		\item $\cM_A[C] = C$, for any $m \times n$ constant matrix $C$.
		\item $\cM_{\beta A}[P] = \beta^k \cM_A[P]$, for any $m \times n$ matrix polynomial $P(\la)$ of grade $k$ and any $\beta \in \FF$.
		\item $\cM_A[\beta P] = \beta \cM_A[P]$.
		\item $\cM_A[P+Q] = \cM_A[P]+\cM_A[Q]$, for any $m \times n$ matrix polynomials $P(\la)$ and $Q(\la)$ both of grade $k$.
		\item Let $P(\la)$ and $Q(\la)$ be two matrix polynomials of grades $k_1$ and $k_2$, respectively. If $P(\la)Q(\la)$ is defined, then $\cM_A[PQ] = \cM_A[P]\cM_A[Q]$, where $P(\la)Q(\la)$ is considered as a matrix polynomial of grade $k_1+k_2$. 
		\item If $Q(\la) = P(\la) \otimes I_n$, then $\cM_A[Q] = \cM_A[P]\otimes I_n$.
		\item $\cM_A[P^\top] = \cM_A[P]^\top$.
		\item If $\FF = \mathbb{C}$, then $\overline{\cM_A[P]} = \cM_{\overline{A}}[\overline{P}]$ and $\cM_A[P]^\ast = \cM_{\overline{A}}[P^\ast] = \cM_{\overline{A}}[\overline{P}]^\top$.
		\item M\"obius transformations act block-wise, i.e., $[\cM_A[P]]_{ij} = \cM_A[P_{ij}]$, for any row and column index sets $i$ and $j$, and where $[P(\la)]_{ij}$ has to be considered as a matrix polynomial with a grade equal to the grade of $P(\la)$.
		\item $\cM_B[\cM_A[P]] = \cM_{AB}[P]$.
		\item Let $P(\la)$ be a minimal basis whose row degrees are all equal to $j$. Then $\cM_A[P]$ is also a minimal basis whose row degrees are all equal to $j$. 
		\item Let $P(\la)$ and $Q(\la)$ be dual minimal bases. If the row degrees of $P(\la)$ are all equal to $j$ and the row degrees of $Q(\la)$ are all equal to $i$, then $\cM_A[P]$ and $\cM_A[Q]$ are also dual minimal bases with all the row degrees of $\cM_A[P]$ equal to $j$ and all the row degrees of $\cM_A[Q]$ equal to $i$.
	\end{enumerate}	
\end{Prop} 

As mentioned before, the advantage of using M\"obius transformations is the possibility of unifying all the known structures of matrix polynomials in a general family. In particular, we are interested in the classes of (skew-)symmetric, (skew-)Hermitian, (anti-)palindromic, and alternating matrix polynomials (they were also studied in detail in \cite{dopico2019structured}). In the following definition, we show the conditions that define any of these structures, relating each of these conditions to the corresponding M\"obius transformation. 

\begin{Def}{\rm\cite[Def. 9.6]{mackey2015mobius}}\label{Structures-Mobius}
{\rm
	Let $A_1 = \left[\begin{smallarray}{cc} 1 & 0 \\ 0 & 1 \end{smallarray}\right]$, $A_2 = \left[\begin{smallarray}{rc} -1 & 0 \\ 0 & 1 \end{smallarray}\right]$, and $A_3 = \left[\begin{smallarray}{cc} 0 & 1 \\ 1 & 0 \end{smallarray}\right]$ (they all belong to $GL(2,\FF)$). An $n\times n$ matrix polynomial $P(\la)$ of grade $k$ is
	\begin{enumerate}[{\rm (a)}]
		\item {\it $\star$-symmetric} if $P^\star(\la) = P(\la) := \cM_{A_1}[P]$,
		\item {\it $\star$-skew-symmetric} if $P^\star(\la) = -P(\la) := -\cM_{A_1}[P]$,
		\item {\it $\star$-even} if $P^\star(\la) = P(-\la) := \cM_{A_2}[P]$,
		\item {\it $\star$-odd} if $P^\star(\la) = -P(-\la) = -\cM_{A_2}[P]$,
		\item {\it $\star$-palindromic} if $P^\star(\la) = {\rm rev}_k P(\la) := \cM_{A_3}[P]$, and
		\item {\it $\star$-anti-palindromic} if $P^\star(\la) = -{\rm rev}_k P(\la) := -\cM_{A_3}[P]$.
	\end{enumerate}
}
\end{Def}
The name {\it $\star$-alternating} is also used for both $\star$-even and $\star$-odd structures. A $*$-symmetric (respectively, $*$-skew-symmetric) matrix polynomial is usually known as {\it Hermitian} (resp., {\it skew-Hermitian}).

In \cite[Def. 3.6]{dopico2019structured}, in order to unify all the identities for the different structures appearing in Definition \ref{Structures-Mobius} in a unique identity, the authors introduced the notion of ``$\cM_A$-structured matrix polynomial" for matrix polynomials of odd grade $k$. To be precise, the matrix polynomial $P(\la)$ is $\cM_A$-structured if it satisfies the identity $\cM_A[P] = P^\star(\la)$. The $\star$-symmetric, $\star$-even, and $\star$-palindromic structures in Definition \ref{Structures-Mobius} satisfy this particular identity, for any grade of $P(\la)$, and for the appropriate matrix $A$. Instead, the $\star$-skew-symmetric, $\star$-odd, and $\star$-anti-palindromic structures satisfy the identity $-\cM_A[P]=P^\star(\la)$, for any grade $k$, which is equivalent to $\cM_A[-P]=P^\star(\la)$, by Proposition \ref{PropersMaPs}-(c).

Notice that, by Proposition \ref{PropersMaPs}-(b), for any odd grade $k$, the identity $\cM_A[-P] = \cM_{-A}[P]$ holds. This allowed the authors of \cite{dopico2019structured} to unify all the structures in Definition \ref{Structures-Mobius} via M\"obius transformations in a unique identity (i.e., $\cM_A[P] = P^\star(\la)$, for appropriate choices of $A$) valid for matrix polynomials of odd grade. The previous identity is no longer true for matrix polynomials of even grade. Therefore, in order to introduce an identity that unifies all the identities in Definition \ref{Structures-Mobius}, for any grade $k$ (even or odd) we need to consider both the $(+)$ and $(-)$ sign. This is done in the following definition, which extends the notion of $\cM_A$-structured matrix polynomial, introduced in \cite[Def. 3.6]{dopico2019structured} for matrix polynomials of odd grade, to matrix polynomials of any grade.

\begin{Def}\label{DefMaStructured}
{\rm
	Let $P(\la) = \sum_{j=0}^k \la^jP_j \in \FF[\la]^{n \times n}$, such that $k \in \mathbb{N}$, and let $A \in GL(2,\FF)$. Then, the matrix polynomial $P(\la)$ is said to be {\it $\cM_A$-structured} if it satisfies
	\begin{equation}\label{Ma-struct}
	\cM_A[P] = P^\star(\la),\ {\rm or}\ \cM_A[-P] = P^\star(\la).
	\end{equation}
}
\end{Def}

In Section \ref{KroneckerStructured} we consider the problem of finding $\ell$-ifications of an $\cM_A$-structured matrix polynomial, in a structure-preserving way, for all the structures in Definition \ref{Structures-Mobius}. Notice that the matrices $A_1$, $A_2$, and $A_3$ in Definition \ref{Structures-Mobius} are coninvolutory. This property will be relevant to guarantee that the constructions are indeed $\cM_A$-structured.

\section{($\boldsymbol{\cM_A}$-structured) block minimal bases matrix polynomials}\label{bmb-sec}

We first review the family of (strong) block minimal bases grade-$\ell$ matrix polynomials introduced in \cite{dopico2019block}. Most of the linearizations, quadratifications and, in general, $\ell$-ifications known so far belong to this wide family (up to permutations). Some examples can be found in \cite[Section 4.1]{dopico2019block}. 

\begin{Def}{\rm\cite[Def. 4.1]{dopico2019block}}\label{SBMB}
{\rm
	A grade-$\ell$ matrix polynomial of the form
	\begin{equation}\label{LBlockMinimal}
	\mathcal{L}(\la) = \begin{bmatrix} M(\la) & K_2^\top(\la) \\ K_1(\la) & 0 \end{bmatrix},
	\end{equation}
	where $M(\la)$ is an arbritrary grade-$\ell$ matrix polynomial, is called a {\it block minimal bases matrix polynomial} if $K_1(\la)$ and $K_2(\la)$ are both minimal bases. If, in addition, the row degrees of $K_1(\la)$ are all equal to $\ell$, the row degrees of $K_2(\la)$ are all equal to $\ell$, the row degrees of a minimal basis dual to $K_1(\la)$ are all equal to each other, and the row degrees of a minimal basis dual to $K_2(\la)$ are all equal to each other, then $\mathcal{L}(\la)$ is called a {\it strong block minimal bases degree-$\ell$ matrix polynomial}.
}
\end{Def}

The following theorem shows that every strong block minimal bases degree-$\ell$ matrix polynomial is always a strong $\ell$-ification of a certain matrix polynomial $P(\la)$. Furthermore, it is also shown that there exists a simple shift relation between the minimal indices of $\mathcal{L}(\la)$ and $P(\la)$.

\begin{Teo}{\rm\cite[Th. 4.2 and 6.2]{dopico2019block}}\label{SLification}
	Let $K_1(\la)$ and $N_1(\la)$ be dual minimal bases, and let $K_2(\la)$ and $N_2(\la)$ be other dual minimal bases. Consider the matrix polynomial
	\begin{equation}\label{Pgeneral}
	P(\la) := N_2(\la)M(\la)N_1^\top(\la),
	\end{equation}
	and the block minimal bases matrix polynomial $\mathcal{L}(\la)$ in \eqref{LBlockMinimal}. Then:
	\begin{enumerate}[{\rm (a)}]
		\item $\mathcal{L}(\la)$ is an $\ell$-ification of $P(\la)$.
		\item If $\mathcal{L}(\la)$ is a strong block minimal bases degree-$\ell$ matrix polynomial, then 
		\begin{enumerate}
			\item[{\rm (b1)}] $\mathcal{L}(\la)$ is a strong $\ell$-ification of $P(\la)$, when $P(\la)$ is considered as a polynomial with grade $\ell + \deg(N_1) + \deg(N_2)$,
			\item[{\rm (b2)}] if $0 \le \epsilon_1 \le \epsilon_2 \le \cdots \le \epsilon_p$ are the right minimal indices of $P(\la)$, then $\epsilon_1 + \deg(N_1) \le \epsilon_2 + \deg(N_1) \le \cdots \le \epsilon_p + \deg(N_1)$ are the right minimal indices of $\mathcal{L}(\la)$, and
			\item[{\rm (b3)}] if $0 \le \eta_1 \le \eta_2 \le \cdots \le \eta_q$ are the left minimal indices of $P(\la)$, then $\eta_1 + \deg(N_2) \le \eta_2 + \deg(N_2) \le \cdots \le \eta_q + \deg(N_2)$ are the left minimal indices of $\mathcal{L}(\la)$.
		\end{enumerate}
	\end{enumerate}
\end{Teo}

In the following remark, we show the relation between the grade and size of the matrix polynomial $P(\la)$ as in \eqref{Pgeneral} and its strong $\ell$-ification $\mathcal{L}(\la)$ as in \eqref{LBlockMinimal} when $P(\la)$ is square.

\begin{Rem}\label{RemGeneral}
{\rm
	Assume that $P(\la):= N_2(\la)M(\la)N_1^\top(\la)$ as in \eqref{Pgeneral} has size $n \times n$ and grade $k = \ell + \deg(N_1) + \deg(N_2):= \ell + \ell_1 + \ell_2$. Let $N_1(\la) \in \FF[\la]^{n\times (n+m_1)}$ and $N_2(\la) \in \FF[\la]^{n \times (n+m_2)}$, so the minimal bases dual to $N_1(\la)$ and $N_2(\la)$ are of the form $K_1(\la) \in \FF[\la]^{m_1 \times (n+m_1)}$ and $K_2(\la) \in \FF[\la]^{m_2 \times (n+m_2)}$, respectively, by Definition {\rm\ref{dualminbases}}-(b). In addition, it must be $\ell m_1 = \ell_1 n$ and $\ell m_2 = \ell_2 n$ (because $K_1(\la)$, $N_1(\la)$, and $K_2(\la)$, $N_2(\la)$, respectively, are dual minimal bases {\rm\cite[Lemma $3.6$-(b)]{de2015matrix}}). Now, adding up these two identities and using $k = \ell + \ell_1 + \ell_2$, we get 
	\begin{equation}\label{sizeLgeneral}
	\ell(m_1 + m_2) = (\ell_1 + \ell_2)n = (k-\ell)n \Leftrightarrow n + m_1 + m_2 = \frac{nk}{\ell}.
	\end{equation}
	Therefore, $\ell$ is a divisor of the product ${\rm size}(P)\cdot {\rm grade}(P)$, for every $\mathcal{L}(\la)$ and $P(\la)$ as in Theorem {\rm\ref{SLification}}. By part {\rm (b1)} in Theorem {\rm\ref{SLification}}, a strong block minimal bases degree-$\ell$ matrix polynomial $\mathcal{L}(\la)$ as in \eqref{LBlockMinimal} is a strong $\ell$-ification of $P(\la)$ and it has size $n+m_1+m_2 = \frac{nk}{\ell}$, by \eqref{sizeLgeneral}.
}
\end{Rem} 

The following definition is a generalization of \cite[Def. 4.3]{dopico2019structured} to grade-$\ell$ matrix polynomials.

\begin{Def}\label{StructSBMB}
{\rm
	Let $K(\la)$ and $N(\la)$ be dual minimal bases, with all the row degrees of $K(\la)$ equal to $\ell$ and with all the row degrees of $N(\la)$ equal to each other, and let $A \in GL(2,\FF)$ be a coninvolutory matrix. Then, an $\cM_A$-structured grade-$\ell$ matrix polynomial of the form 
	\begin{equation}\label{StructLSBlock}
	\begin{split}
	\mathcal{L}(\la) &= \begin{bmatrix} M(\la) & \cM_A[\pm K]^\star \\ K(\la) & 0 \end{bmatrix} = \begin{bmatrix} M(\la) & \cM_{\overline{A}}[\pm \overline{K}]^\top \\ K(\la) & 0 \end{bmatrix} \ {\rm with\ } \cM_A[\pm M] = M^\star(\la),
	\end{split}
	\end{equation}
	is called an {\it $\cM_A$-structured strong block minimal bases degree-$\ell$ matrix polynomial}.
}
\end{Def}

Notice that $\mathcal{L}(\la)$ as in \eqref{StructLSBlock} satisfies the identity $\cM_A[\pm \mathcal{L}] = \mathcal{L}^\star(\la)$ (Definition \ref{DefMaStructured}) if and only if $A$ is coninvolutory. In particular, 
\begin{equation*}
\begin{array}{ccl}
\cM_A[\pm \mathcal{L}] &:=& \begin{bmatrix} \cM_A[\pm M] & \cM_A\left[\pm \cM_{A}[\pm K]^\star\right] \\ \cM_A[\pm K] & 0 \end{bmatrix} = \begin{bmatrix} \cM_A[\pm M] & \cM_A\left[\pm \cM_{\overline{A}}[\pm \overline{K}]^\top\right] \\ \cM_A[\pm K] & 0 \end{bmatrix} \\
&=& \begin{bmatrix} \cM_A[\pm M] & \cM_A\left[\pm \cM_{\overline{A}}[\pm \overline{K}]\right]^\top \\ \cM_A[\pm K] & 0 \end{bmatrix} = \begin{bmatrix} \cM_A[\pm M] & \cM_A\left[\cM_{\overline{A}}[\overline{K}]\right]^\top \\ \cM_A[\pm K] & 0 \end{bmatrix} \\
&=& \begin{bmatrix} \cM_A[\pm M] & \cM_{\overline{A}A}[\overline{K}]^\top \\ \cM_A[\pm K] & 0 \end{bmatrix} = \begin{bmatrix} \cM_A[\pm M] & \overline{K}^\top(\la) \\ \cM_A[\pm K] & 0 \end{bmatrix} = \begin{bmatrix} M^\star(\la) & K^\star(\la) \\ \cM_{A}[\pm K]  & 0 \end{bmatrix} =: \mathcal{L}^\star(\la).
\end{array}
\end{equation*}
To get the previous identity it is important to note that the $\pm$ sign in $\pm \mathcal{L}(\la)$ at the begining is in accordance with the $\pm$ sign that appears in \eqref{StructLSBlock}. For this reason, the signs cancel out. 

Therefore, any $\cM_A$-structured strong block minimal bases degree-$\ell$ matrix polynomial is, as the name suggests, $\cM_A$-structured.

As a corollary of Theorem \ref{SLification}, any $\cM_A$-structured strong block minimal bases degree-$\ell$ matrix polynomial is always a strong $\ell$-ification of a certain $\cM_A$-structured matrix polynomial. This is a straightforward extension of \cite[Th. 4.5]{dopico2019structured} for grade-$\ell$ matrix polynomials.

\begin{Teo}\label{StructLification}
	Let $K(\la)$ and $N(\la)$ be dual minimal bases, with all the row degrees of $K(\la)$ equal to $\ell$ and with all the row degrees of $N(\la)$ equal to each other. Let $A \in GL(2,\FF)$ be a coninvolutory matrix, and let $\mathcal{L}(\la)$ be an $\cM_A$-structured strong block minimal bases degree-$\ell$ matrix polynomial as in \eqref{StructLSBlock}. Then, the matrix polynomial $\mathcal{L}(\la)$ is a strong $\ell$-ification of the $\cM_A$-structured matrix polynomial
	\begin{equation}\label{StructPlambda}
	P(\la) := \left(\cM_{\overline{A}}[\pm \overline{N}]MN^\top\right)(\la),
	\end{equation}
	of grade $\ell + 2\deg(N)$. Moreover,
	\begin{itemize}
		\item if $0 \le \epsilon_1 \le \epsilon_2 \le \cdots \le \epsilon_p$ are the right minimal indices of $P(\la)$, then $\epsilon_1 + \deg(N) \le \epsilon_2 + \deg(N) \le \cdots \le \epsilon_p + \deg(N)$ are the right minimal indices of $\mathcal{L}(\la)$, and
		\item if $0 \le \eta_1 \le \eta_2 \le \cdots \le \eta_p$ are the left minimal indices of $P(\la)$, then $\eta_1 + \deg(N) \le \eta_2 + \deg(N) \le \cdots \le \eta_p + \deg(N)$ are the left minimal indices of $\mathcal{L}(\la)$.
	\end{itemize}
\end{Teo}
Although the proof is very similar to the one of \cite[Th. 4.5]{dopico2019structured}, we include it here for completeness.

{\em Proof (of Theorem {\em \ref{StructLification}})}.
%\begin{proof}
	Notice that $K(\la)N^\top(\la) = 0$ implies $\overline{K}(\la)\overline{N}^\top(\la) = 0$. Then, we have that $\pm \overline{K}(\la)$ and $\pm \overline{N}(\la)$ are dual minimal bases with all the row degrees of $\pm \overline{K}(\la)$ equal to $\ell$, and all the row degrees of $\pm \overline{N}(\la)$ equal to $\deg(N)$. Then, by Proposition \ref{PropersMaPs}-(k), we conclude that $\cM_{\overline{A}}[\pm \overline{K}]$ and $\cM_{\overline{A}}[\pm \overline{N}]$ are also dual minimal bases with all the row degrees of $\cM_{\overline{A}}[\pm \overline{K}]$ equal to $\ell$, and all the row degrees of $\cM_{\overline{A}}[\pm \overline{N}]$ equal to $\deg(N)$. Therefore, $\mathcal{L}(\la)$ is a strong block minimal bases degree-$\ell$ matrix polynomial. By Theorem \ref{SLification}, we immediately obtain that $\mathcal{L}(\la)$ is a strong $\ell$-ification of $P(\la)$ and that the minimal indices of $\mathcal{L}(\la)$ are those of $P(\la)$ shifted by $\deg(N)$. We still have to show that $P(\la)$ is $\cM_A$-structured, that is, $\cM_A[\pm P] = P^\star(\la)$. Computing the M\"obius transformation of $P(\la)$ associated to the matrix $A$ and using that $A$ is coninvolutory, together with parts {\rm (c)}, {\rm (e)}, {\rm (g)}, {\rm (h)}, {\rm (j)} in Proposition \ref{PropersMaPs}, we get 
	\begin{equation*}
	\begin{array}{ccl}
	\cM_A[\pm P] &=& \cM_A\left[\pm \cM_{\overline{A}}[\pm \overline{N}]MN^\top\right] = \cM_A\left[\cM_{\overline{A}}[\pm \overline{N}]\right]\cM_A\left[\pm M\right]\cM_A\left[N^\top\right] \\
	&=& \pm \overline{N} M^\star \cM_A[N]^\top = (N^\top)^\star M^\star \cM_{\overline{A}}[\pm \overline{N}]^\star = P^\star(\la). ~~~ \cvd
	\end{array} 
	\end{equation*} 
%\end{proof}

For structured matrix polynomials, we rewrite Remark \ref{RemGeneral} as follows.

\begin{Rem}\label{RemStruct}
{\rm
	Let $P(\la) \in \FF[\la]^{n \times n}$ be a matrix polynomial of grade $k=\ell + 2\deg(N):=\ell + 2\ell_1$ defined as in \eqref{StructPlambda}, where $N(\la) \in \FF[\la]^{n \times (n+m)}$ and $K(\la) \in \FF[\la]^{m \times (n+m)}$ are dual minimal bases as in Theorem \ref{StructLification}. Note that this implies $k-\ell = 2\ell_1$, so $k$ and $\ell$ (the grades of $P(\la)$ as in \eqref{StructPlambda} and $\mathcal{L}(\la)$ as in \eqref{StructLSBlock}, respectively) have the same parity. Then, $\ell m = \ell_1 n$, by {\rm\cite[Lemma $3.6$-(b)]{de2015matrix}}, and replacing $\ell_1 = \frac{k-\ell}{2}$ in this identity, we get
	\begin{equation}\label{sizeLstruct}
	\ell m = \frac{k-\ell}{2}n \Leftrightarrow n + 2m = \frac{nk}{\ell}.
	\end{equation}
	Now, $\mathcal{L}(\la) \in \FF[\la]^{(n+2m) \times (n+2m)}$ as in \eqref{StructLSBlock} is a strong $\ell$-ification of $P(\la)$ by Theorem {\rm\ref{StructLification}}, and, by \eqref{sizeLstruct}, it has size $n+2m = \frac{nk}{\ell}$.
}
\end{Rem}

Notice that, in order to look for a unified framework of $\cM_A$-structured matrix polynomials, the grade parity has to be taken into account, as a consequence of Remark \ref{RemStruct}. Therefore, we will analyze both cases on the grade parity separately (namely, $k,\ell$ even and $k,\ell$ odd) when necessary.

The following theorem is a natural extension of \cite[Th. 4.6]{dopico2019structured} to grade-$\ell$ matrix polynomials. Now, we consider the inverse problem, that is, given an $\cM_A$-structured matrix polynomial $P(\la)$, we show how to construct a structure-preserving strong $\ell$-ification of $P(\la)$. This is, actually, the relevant problem we are interested in, since usually the given input is the matrix polynomial $P(\la)$. Recall that the grades of $P(\la)$ and the strong $\ell$-ification must have the same parity, by Remark \ref{RemStruct}.

\begin{Teo}\label{Ma-strucstrong}
	Let $K(\la)$ and $N(\la)$ be dual minimal bases, with all the row degrees of $K(\la)$ equal to $\ell$ and with all the row degrees of $N(\la)$ equal to each other, and let $A \in GL(2,\FF)$ be a coninvolutory matrix. Let $P(\la)$ be a given $\cM_A$-structured matrix polynomial of grade $k:=\ell + 2\deg(N)$. If $M(\la)$ is any solution of the equation \eqref{StructPlambda} (not necessarily $\cM_A$-structured), then the grade-$\ell$ matrix polynomial
	\begin{equation*}
	\widetilde{M}(\la):= \frac{1}{2}\left(M + \cM_A[\pm M]^\star\right)(\la)
	\end{equation*}
	is an $\cM_A$-structured solution of \eqref{StructPlambda}, and the $\cM_A$-structured strong block minimal bases degree-$\ell$ matrix polynomial
	\begin{equation}\label{MaL1/2}
	\mathcal{L}(\la) := \begin{bmatrix} \widetilde{M}(\la) & \cM_A[\pm K]^\star \\ K(\la) & 0 \end{bmatrix} = \begin{bmatrix} \frac{1}{2}\left(M + \cM_A[\pm M]^\star\right)(\la) & \cM_A[K]^\star \\ K(\la) & 0 \end{bmatrix}
	\end{equation}
	is a strong $\ell$-ification of $P(\la):= \left(\cM_{\overline{A}}[\pm \overline{N}]\widetilde{M}N^\top\right)(\la)$. 
\end{Teo}

\begin{proof}
	Since $M(\la)$ is, by hypothesis, a solution of \eqref{StructPlambda}, in order to show that $\widetilde{M}(\la)$ is a solution of \eqref{StructPlambda} it suffices to show that $\cM_A[\pm M]^\star$ is also a solution of \eqref{StructPlambda}, since any affine combination of solutions of \eqref{StructPlambda} is also a solution of \eqref{StructPlambda}. Using that $P(\la)$ is $\cM_A$-structured, together with the fact that $A$ is coninvolutory, as well as parts {\rm (c)}, {\rm (e)}, {\rm (g)}, {\rm (h)}, {\rm (j)} in Proposition \ref{PropersMaPs}, we have
	\begin{equation*}
	\begin{split}
	P(\la) =& \cM_A[\pm P]^\star = \cM_A\left[\pm \cM_{\overline{A}}[\pm \overline{N}]MN^\top\right]^\star = \left(\pm \overline{N}\cM_A[\pm M] \cM_A[N]^\top\right)^\star \\
	=& \left( \overline{N}\cM_A[\pm M] \cM_A[\pm N]^\top\right)^\star =  \cM_{\overline{A}}[\pm \overline{N}]\cM_A[\pm M]^\star N^\top.
	\end{split}
	\end{equation*}
	Furthermore, notice that the matrix polynomial $\widetilde{M}(\la)$ is $\cM_A$-structured, because
	\begin{equation*}
	\cM_A[\pm \widetilde{M}] := \cM_A\left[\pm \frac{1}{2}\left(M + \cM_A[\pm M]^\star\right)\right] = \frac{1}{2}\left(\cM_A[\pm M] + M^\star\right) =: \widetilde{M}^\star(\la),
	\end{equation*}
	then $\mathcal{L}(\la)$ as in \eqref{MaL1/2} is an $\cM_A$-structured strong block minimal bases degree-$\ell$ matrix polynomial.
\end{proof}

It is important to mention that we can always construct an $\cM_A$-structured strong block minimal bases degree-$\ell$ matrix polynomial as in \eqref{MaL1/2}, which is a strong $\ell$-ification of a certain matrix polynomial, not necessarily $\cM_A$-structured. However, to guarantee that it is a strong $\ell$-ification of $P(\la) := \left(\cM_{\overline{A}}[\pm \overline{N}]\widetilde{M}N^\top\right)(\la)$, then $P(\la)$ must be $\cM_A$-structured (see the proof of Theorem \ref{Ma-strucstrong}). We refer the reader to Example \ref{ExamDuplic} (at the end of the example) for more details.

Theorem \ref{Ma-strucstrong} provides a way to construct a family of $\cM_A$-structured strong $\ell$-ifications of a given $\cM_A$-structured matrix polynomial $P(\la)$. However, this family depends on two ingredients: {\rm (a)} a matrix polynomial $M(\la)$, which is a solution of \eqref{StructPlambda}, and {\rm (b)} a pair of dual minimal bases satisfying the conditions of the statement of Theorem \ref{Ma-strucstrong}. In the next section, we will show how to obtain such an $M(\la)$ from the coefficients of $P(\la)$ using some particular dual minimal bases (namely, ``block-Kronecker dual minimal bases"). This will give us an $\cM_A$-structured generalized companion $\ell$-ification of any given $\cM_A$-structured matrix polynomial of grade $k$. The construction is valid for $k=s \ell$, with $s$ being an odd number.

\section{$\boldsymbol{\cM_A}$-structured block-Kronecker matrix polynomials}\label{KroneckerStructured}

We focus on the problem of constructing explicitly structure-preserving strong $\ell$-ifications for all structures in Definition \ref{Structures-Mobius}. In particular, we will construct them from a subfamily of $\cM_A$-structured strong block minimal bases degree-$\ell$ matrix polynomials (see Definition \ref{StructBK}).

In \cite[Def. 5.6]{dopico2019block}, the authors introduced the notion of {\it $(\epsilon,n,\eta,m)$-block-Kronecker degree-$\ell$ matrix polynomial}. We adapt this definition, and the corresponding results, to square matrix polynomials. Then, an {\it $(\epsilon,n,\epsilon,n)$-block-Kronecker degree-$\ell$ matrix polynomial} is of the form
\begin{equation*}
\mathcal{L}(\la) := \begin{bmatrix} M(\la) & L_\epsilon^\top(\la^\ell) \otimes I_n \\ L_\epsilon(\la^\ell) \otimes I_n & 0 \end{bmatrix}.
\end{equation*}
It is proved in \cite[Th. 5.7]{dopico2019block} that any $(\epsilon,n,\epsilon,n)$-block-Kronecker degree-$\ell$ matrix polynomial is a strong $\ell$-ification of the matrix polynomial 
\begin{equation}\label{recoverpoly}
P(\la) = (\Lambda_\epsilon^\top(\la^\ell) \otimes I_n) M(\la) (\Lambda_\epsilon(\la^\ell) \otimes I_n)
\end{equation}
of grade $(2\epsilon + 1)\ell$. Notice that this family of matrix polynomials is a subfamily of the strong block minimal bases matrix polynomials (Definition \ref{SBMB}). Furthermore, this family generalizes the block-Kronecker pencils in \cite{dopico2018block} from $\ell=1$ to any degree $\ell$. The main advantage of using this subfamily for constructing strong $\ell$-ifications is that it is very easy to construct the $(1,1)$ big block $M(\la)$ in order to recover the given matrix polynomial $P(\lambda)$ as in \eqref{recoverpoly}, due to the choice of the specific block-Kronecker dual minimal bases (Theorems \ref{NoStructLlambdaTeo} and \ref{StructLgeneral}). In particular, we only have to look for an appropriate matrix polynomial $M(\la)$ that allows us to recover the polynomial $P(\lambda)$ in \eqref{recoverpoly}, to construct block-Kronecker degree-$\ell$ matrix polynomials. However, since we are also looking for structure-preserving strong $\ell$-ifications, we introduce, in Definition \ref{StructBK}, the notion of ``$\cM_A$-structured block-Kronecker degree-$\ell$ matrix polynomial", extending the one in \cite[Def. 5.1]{dopico2019structured}, valid for matrix pencils. This is a subfamily of the $\cM_A$-structured strong block minimal bases degree-$\ell$ matrix polynomials in Definition \ref{StructSBMB} obtained when particularizing the minimal basis $K(\la)$ to be of the form $L_d(\la^\ell) \otimes I_n$, with $L_d(\la)$ being a Kronecker pencil as in \eqref{Ld}.

\begin{Def}\label{StructBK}
{\rm
	Let $L_d(\la) \in \FF[\la]^{d \times (d+1)}$ be the matrix pencil defined in \eqref{Ld}, and let $A \in GL(2,\FF)$ be a coninvolutory matrix. Then a degree-$\ell$ matrix polynomial of the form 
	\begin{equation}\label{StructLSBlockKronecker}
	\mathcal{L}(\la) = \begin{bmatrix} M(\la) & \cM_A[\pm L_d]^\top(\la^\ell) \otimes I_n \\ L_d(\la^\ell) \otimes I_n & 0 \end{bmatrix}, \ {\rm with\ } \cM_A[\pm M] = M^\star(\la),
	\end{equation}
	is called an {\it $\cM_A$-structured block-Kronecker degree-$\ell$ matrix polynomial}.
}
\end{Def}

The following theorem is a corollary of Theorem \ref{StructLification} and, then, a simple extension of \cite[Theorem 5.3]{dopico2019structured}.

\begin{Teo}\label{StructPLambdaKroneckerTeo}
	Let $\mathcal{L}(\la)$ be an $\cM_A$-structured block-Kronecker degree-$\ell$ matrix polynomial as in \eqref{StructLSBlockKronecker}. Then, the matrix polynomial $\mathcal{L}(\la)$ is a strong $\ell$-ification of the $\cM_A$-structured matrix polynomial of grade $(2d + 1)\ell$
	\begin{equation}\label{StructPlambdaKronecker}
	P(\la) := (\cM_{A}[\pm \Lambda_d]^\top(\la^\ell) \otimes I_n)M(\la)(\Lambda_d(\la^\ell) \otimes I_n),
	\end{equation}
	where $\Lambda_d(\la^\ell)$ is as in \eqref{Lambdad}. Moreover, 
	\begin{itemize}
		\item if $0 \le \epsilon_1 \le \epsilon_2 \le \cdots \le \epsilon_p$ are the right minimal indices of $P(\la)$, then $\epsilon_1 + d\ell \le \epsilon_2 + d\ell \le \cdots \le \epsilon_p + d\ell$ are the right minimal indices of $\mathcal{L}(\la)$, and
		\item if $0 \le \eta_1 \le \eta_2 \le \cdots \le \eta_p$ are the left minimal indices of $P(\la)$, then $\eta_1 + d\ell \le \eta_2 + d\ell \le \cdots \le \eta_p + d\ell$ are the left minimal indices of $\mathcal{L}(\la)$.
	\end{itemize}
\end{Teo}

%For simplicity in the reading, we omit the expression $\la^\ell$ when we present the matrix polynomials in a general way. We will only explicitly write it when we show examples of matrix polynomials with concrete degree $\ell$. Then, from now on, $L_d := L_d(\la^\ell)$, $\Lambda_d := \Lambda_d(\la^\ell)$. 

As in Remarks \ref{RemGeneral} and \ref{RemStruct}, we show, in the following remark, the relation between the size and the grade of the polynomials $\mathcal{L}(\la)$ and $P(\la)$ in \eqref{StructLSBlockKronecker} and \eqref{StructPlambdaKronecker}, respectively, and we deduce the relation on the parity of both grades.

\begin{Rem}
	Let $P(\la)$ be the matrix polynomial in \eqref{StructPlambdaKronecker}, with grade $k:= (2d+1)\ell$ and size $n \times n$. Then, $\mathcal{L}(\la)$ in \eqref{StructLSBlockKronecker} has size $(2d+1)n \times (2d+1)n$. Note that $\ell$ is a divisor of $k$, that is, the degree of $\mathcal{L}(\la)$ as in \eqref{StructLSBlockKronecker} is a divisor of the grade of $P(\la)$ as in \eqref{StructPlambdaKronecker}. Moreover, $k$ and $\ell$ have the same parity, since $2d+1$ is odd. 
\end{Rem}

In Table \ref{tableMa} we list the minimal bases $\cM_A[\pm L_d](\la^\ell)$ and the conditions on the grade-$\ell$ matrix polynomial $M(\la)$ for $\cM_A$-structured block-Kronecker degree-$\ell$ matrix polynomials, according to Definition \ref{Structures-Mobius}. Notice that the parity of $\ell$ is only relevant for the $\star$-even and $\star$-odd structures. 
\begin{table}[H]
	\centering
	\resizebox{16.5cm}{3cm}{
	\begin{tabular}{| c | c | c | c |}
		\hline
		Structure & $\cM_A[\pm L_d](\la^\ell)$ & Conditions on $M(\la) = \sum_{i=0}^\ell \la^i M_i$ & Parity of $\ell$ \\
		\hline
		\hline
		$\star$-symmetric & $\cM_{A_1}[L_d](\la^\ell) = L_d(\la^\ell)$ & $M_i^\star = M_i$, for $i=0,1,\hdots,\ell$ & any $\ell$ \\
		\hline
		$\star$-skew-symmetric & $\cM_{A_1}[-L_d](\la^\ell) =-L_d(\la^\ell)$ & $M_i^\star = -M_i$, for $i=0,1,\hdots,\ell$ & any $\ell$ \\
		\hline
		\multirow{2}{2.5cm}{$\begin{array}{c} \textup{$\star$-even} \\  \textup{(alternating)}\end{array}$} & $\cM_{A_2}[L_d](\la^\ell) = L_d(-\la^\ell)$ & $\begin{array}{c} {M_{2i}}^\star = M_{2i} \ {\rm and}\ {M_{2i+1}}^\star = -M_{2i+1}, \\ {\rm for}\ i=0,1,\hdots,\frac{\ell-1}{2}\end{array}$ & $\ell$ odd \\ \cline{2-4} & $\cM_{A_2}[L_d](\la^\ell) = L_d(\la^\ell)$ & $\begin{array}{c} {M_{2i}}^\star = M_{2i},\ {\rm for}\ i=0,1,\hdots,\frac{\ell}{2},\ {\rm and} \\ {M_{2i+1}}^\star = -M_{2i+1},\ {\rm for}\ i=0,1,\hdots,\frac{\ell}{2}-1\end{array}$ & $\ell$ even \\ 
		\hline
		\multirow{2}{2.5cm}{$\begin{array}{c} \textup{$\star$-odd} \\ \textup{(alternating)}\end{array}$} & $\cM_{A_2}[-L_d](\la^\ell) =-L_d(-\la^\ell)$ & $\begin{array}{c} {M_{2i}}^\star = -M_{2i} \ {\rm and}\ {M_{2i+1}}^\star = M_{2i+1}, \\ {\rm for}\ i=0,1,\hdots,\frac{\ell-1}{2}\end{array}$ & $\ell$ odd \\ \cline{2-4} & $\cM_{A_2}[-L_d](\la^\ell) = -L_d(\la^\ell)$ &  $\begin{array}{c} {M_{2i}}^\star = -M_{2i},\ {\rm for}\ i=0,1,\hdots,\frac{\ell}{2},\ {\rm and} \\ {M_{2i+1}}^\star = M_{2i+1},\ {\rm for}\ i=0,1,\hdots,\frac{\ell}{2}-1\end{array}$ & $\ell$ even \\
		\hline
		$\star$-palindromic & $\cM_{A_3}[L_d](\la^\ell) = {\rm rev}_\ell L_d(\la^\ell)$ & $M_i^\star = M_{\ell-i}$, for $i=0,1,\hdots,\lceil \frac{\ell-1}{2}\rceil$ & any $\ell$\\
		\hline
		$\star$-anti-palindromic & $\cM_{A_3}[-L_d](\la^\ell) = -{\rm rev}_\ell L_d(\la^\ell)$ & $M_i^\star = -M_{\ell-i}$, for $i=0,1,\hdots,\lceil \frac{\ell-1}{2}\rceil$ & any $\ell$\\
		\hline
	\end{tabular}}
	\caption{Minimal bases $\cM_A[\pm L_d](\la^\ell)$ and conditions on the grade-$\ell$ matrix polynomial $M(\la):= \sum_{i=0}^\ell \la^i M_i$, taking into account the parity of $\ell$ when necessary.}
	\label{tableMa}
\end{table}

Definition \ref{structures} introduces three different conditions on the matrix polynomial $M(\la)$, extending the ones in \cite[Th. 5.4]{dopico2019structured} for matrix pencils. They will be used to characterize (structured) strong $\ell$-ifications of the given (structured) matrix polynomial $P(\la)$ as in \eqref{mpoly}. 

\begin{Def}\label{structures}
{\rm
	Let $P(\la) = \sum_{j=0}^k \la^j P_j \in \FF[\la]^{n \times n}$ be a grade-$k$ matrix polynomial, and let $M(\la) = \sum_{i=0}^\ell \la^i M_i \in \FF[\la]^{(d+1)n \times (d+1)n}$ be a grade-$\ell$ matrix polynomial, with $k=(2d+1)\ell$, for some $d \geq 0$. Let us partition the matrices $M_i$, for $i=0,1,\hdots,\ell$, into $(d+1) \times (d+1)$ blocks, each of size $n \times n$, and let us denote the blocks by $[M_i]_{s,t} \in \FF^{n \times n}$, for $s,t = 1,2,\hdots,d+1$, and $i=0,1,\hdots,\ell$. Then, we say that $M(\la)$ satisfies
	\begin{enumerate}
		\item[{\rm (AS)}] the {\it anti-diagonal sum condition} for $P(\la)$ if
	\end{enumerate}
	$\left\{\begin{array}{ccll} P_{\ell r} &= &\displaystyle\sum_{s+t = 2d+2-r} [M_0]_{s,t} + \displaystyle\sum_{s+t = 2d+3-r} [M_\ell]_{s,t} & {\rm for}\ r = 0,1,\hdots,2d+1, \\ P_{\ell r+c} &= & \displaystyle\sum_{s+t = 2d+2-r} [M_c]_{s,t} & \begin{array}{l}{\rm for}\ r = 0,1,\hdots,2d, \\{\rm and}\ c = 1,2,\hdots,\ell-1. \end{array} \end{array}\right.$
	\begin{enumerate}
		\item[{\rm (ASS)}] the {\it anti-diagonal signed sum condition} for $P(\la)$ if
	\end{enumerate}
	$\left\{\begin{array}{ccll} P_{\ell r}& = &\displaystyle\sum_{s+t = 2d+2-r} (-1)^{d-s+1}[M_0]_{s,t} + \displaystyle\sum_{s+t = 2d+3-r} (-1)^{d-s+1}[M_\ell]_{s,t} & {\rm for}\ r = 0,1,\hdots,2d+1,\ {\rm and} \\ P_{\ell r+c}& = &\displaystyle\sum_{s+t = 2d+2-r} (-1)^{d-s+1}[M_c]_{s,t} & \begin{matrix}{\rm for}\ r = 0,1,\hdots,2d, \\{\rm and}\ c = 1,2,\hdots,\ell-1. \end{matrix} \end{array}\right.$
	\begin{enumerate}
		\item[{\rm (DS)}] the {\it diagonal sum condition} for $P(\la)$ if
	\end{enumerate}
	$\left\{\begin{array}{ccll} P_{\ell r} &=& \displaystyle\sum_{s-t = r-d} [M_0]_{s,t} + \displaystyle\sum_{s-t = r-d-1} [M_\ell]_{s,t} & {\rm for}\ r = 0,1,\hdots,2d+1,\ {\rm and} \\ P_{\ell r+c} &= &\displaystyle\sum_{s-t = r-d} [M_c]_{s,t} & \begin{matrix}{\rm for}\ r = 0,1,\hdots,2d, \\{\rm and}\ c = 1,2,\hdots,\ell-1. \end{matrix} \end{array}\right.$
}
\end{Def}

The {\rm AS} condition was introduced for the first time in \cite{bueno2018explicit} for symmetric matrix pencils. The names and acronyms of the other conditions have been chosen in order to be consistent with this one. 

Note that the {\rm AS} condition states that the sum of the blocks on the $(2d+1-r)$th block anti-diagonal of $M_0$ plus the sum of the blocks on the $(2d+2-r)$th block anti-diagonal of $M_\ell$ must be equal to the coefficient $P_{\ell r}$ of $P(\la)$, and that the sum of the remaining blocks on the $(2d+1-r)$th block anti-diagonal of $M_c$, for $c=1,2,\hdots,\ell-1,$ must be equal to the coefficient $P_{\ell r + c}$.

The {\rm ASS} condition is similar to the {\rm AS} condition but adding some minus signs in certain blocks, depending on the respective rows they are placed in, together with the parity of $d$. More precisely,
\begin{equation*}
\left\{\begin{array}{cl} (-1)^{d-s+1} = -1 & \left\{ \begin{array}{l} {\rm if\ } d {\rm \ is\ even,\ in\ even\ rows,} \\ {\rm if\ } d {\rm \ is\ odd,\ in\ odd\ rows,} \end{array}\right. \\ (-1)^{d-s+1} = 1, & {\rm otherwise.} \end{array}\right.
\end{equation*}

Finally, the {\rm DS} condition states that the sum of the blocks on the $(2d+1-r)$th block diagonal of $M_0$ plus the sum of the blocks on the $(2d+2-r)$th block diagonal of $M_\ell$ must be equal to the coefficient $P_{\ell r}$ of $P(\la)$, and that the sum of the remaining blocks on the $(2d+1-r)$th block diagonal of $M_c$, for $c=1,2,\hdots,\ell-1,$ must be equal to the coefficient $P_{\ell r + c}$, where the order of the diagonals increases from the lower left corner (the block entry $(d+1,1)$) to the upper right corner (the block entry $(1,d+1)$).

Notice that, when $\ell=1$, the second line in each of the {\rm AS}, {\rm ASS}, or {\rm DS} conditions does not appear in the corresponding conditions (just look at the range of $c$). In this case, $M(\la)$ is a matrix pencil of the form $M(\la) = M_0 + \la M_1$, and only the first line of these conditions applies. In particular, for matrix pencils, these conditions are introduced in \cite[Th. 5.4]{dopico2019structured}.

Regarding $\ell>1$, as a consequence of the second line in any of the AS, ASS, and DS conditions, there is, at least, one nonzero block entry for each (anti-)diagonal (there is a total number of $2d+1$ (anti-)diagonals), because the coefficient $P_{\ell r + c}$, for $c=1,2,\hdots,\ell-1$, and $r=0,1,\hdots,2d$, must be equal to the sum of all the blocks of $M_c$ placed on the $(2d+1-r)$th (anti-)diagonal.

In Theorem \ref{NoStructLlambdaTeo}, we show how to construct strong $\ell$-ifications of a given matrix polynomial of grade $k$, using the conditions {\rm AS}, {\rm ASS}, and {\rm DS} introduced in Definition \ref{structures}. In particular, we impose these conditions to the $(1,1)$ big block of a degree-$\ell$ matrix polynomial as in \eqref{StructLSBlockKronecker}. We emphasize that the given matrix polynomial and its $\ell$-ification are not necessarily $\cM_A$-structured, although the basic template of the $\ell$-ification is the same as for the $\cM_A$-stuctured block-Kronecker degree-$\ell$ matrix polynomial in \eqref{StructLSBlockKronecker}. The reason for this is to obtain $\cM_A$-structured strong $\ell$-ifications when the $(1,1)$ block is $\cM_A$-structured (which is provided in Theorem \ref{StructLgeneral}).

In the forthcoming developments, we will apply M\"obius transformations to several different matrix polynomials, say $Q(\la)$ (in particular, $Q(\la)$ is either $M(\la)$, $L_d(\la^\ell)$, $\Lambda_d(\la^\ell)$, $P(\la)$, or $\mathcal{L}(\la)$). In some cases, the M\"obius transformations are applied to $Q(\la)$ but, in some other cases, they are applied to $-Q(\la)$. In the first case, we will use the notation $\cM_A[+]$, where $A$ is the matrix associated to the M\"obius transformation, and in the second case, we use $\cM_A[-]$. %To unify, we will use in some cases the notation $M_{A}[\pm] \in \{M_{A}[+],\ M_{A}[-]\}$.

We also have to take into account the parity of the grades $k$, $\ell$ in Theorem \ref{NoStructLlambdaTeo} (recall that $k$ and $\ell$ have the same parity).

\begin{Teo}\label{NoStructLlambdaTeo}
	Let $P(\la) = \sum_{j=0}^k \la^j P_j \in \FF[\la]^{n \times n}$ be a matrix polynomial of grade $k$ (not necessarily $\cM_A$-structured), and let $M(\la) = \sum_{i=0}^\ell \la^i M_i \in \FF[\la]^{(d+1)n \times (d+1)n}$ be a grade-$\ell$ matrix polynomial, with $k=(2d+1)\ell$, for some $d \geq 0$. Let us partition the matrices $M_i$, for $i=0,1,\hdots,\ell$, into $(d+1) \times (d+1)$ blocks, each of size $n \times n$, and let us denote the blocks by $[M_i]_{s,t} \in \FF^{n \times n}$ for $s,t = 1,2,\hdots,d+1$, and $i=0,1,\hdots,\ell$. Let $A$ be one of the matrices in Definition {\rm\ref{Structures-Mobius}} and set 
	\begin{equation}\label{NoStructLlambda}
	\mathcal{L}(\la) = \begin{bmatrix} M(\la) & \cM_A[\pm L_d]^\top(\la^\ell) \otimes I_n \\ L_d(\la^\ell) \otimes I_n & 0 \end{bmatrix}.
	\end{equation}
	Then, if $M(\la)$ satisfies either
	\begin{enumerate}[{\rm (i)}]
		\item the {\rm AS} condition for $P(\la)$, when $\left\{\begin{array}{ll} A=A_1, & for\ any\ \ell,\ or \\ A=A_2, & if\ \ell\ is\ even,\ or \end{array}\right.$
		%\item the {\rm AS} condition for $-P(\la)$, for $M_A[\pm] \in \{ M_{A_2}[-]\}$, if $\ell$ is even, or
		\item the {\rm ASS} condition for $P(\la)$, when $A=A_2$, if $\ell$ is odd, or 
		\item the {\rm DS} condition for $P(\la)$, when $A=A_3$, for any $\ell$,
	\end{enumerate}
	the matrix polynomial $\mathcal{L}(\la)$ in \eqref{NoStructLlambda} is a strong block minimal bases degree-$\ell$ matrix polynomial (not necessarily $\cM_A$-structured) such that:
	\begin{enumerate}[{\rm (a)}]
		\item $\mathcal{L}(\la)$ is a strong $\ell$-ification of $P(\la)$,
		\item if $0 \le \epsilon_1 \le \epsilon_2 \le \cdots \le \epsilon_p$ are the right minimal indices of $P(\la)$, then $\epsilon_1 + d\ell \le \epsilon_2 + d\ell \le \cdots \le \epsilon_p + d\ell$ are the right minimal indices of $\mathcal{L}(\la)$, and
		\item if $0 \le \eta_1 \le \eta_2 \le \cdots \le \eta_p$ are the left minimal indices of $P(\la)$, then $\eta_1 + d\ell \le \eta_2 + d\ell \le \cdots \le \eta_p + d\ell$ are the left minimal indices of $\mathcal{L}(\la)$.
	\end{enumerate}
\end{Teo}

\begin{proof} Lemma \ref{dualminbasesell}-(b) and Proposition \ref{PropersMaPs}-(l) together imply that $\mathcal{L}(\la)$ is a strong block minimal bases degree-$\ell$ matrix polynomial (Definition \ref{SBMB}).
	
	Part {\rm (b1)} in Theorem \ref{SLification} guarantees that $\mathcal{L}(\la)$ is a strong $\ell$-ification of the matrix polynomial $Q(\la):= (\cM_A[\pm \Lambda_d]^\top(\la^\ell) \otimes I_n)M(\la)(\Lambda_d(\la^\ell) \otimes I_n)$. Then, to prove part {\rm (a)}, we just need to check that $Q(\la) = P(\la)$ or $Q(\la) = -P(\la)$, where the sign depends on the M\"obius transformation $\cM_A$. In particular, the identity $Q(\la) = P(\la)$ holds for $\cM_A[+]$ and the identity $Q(\la) = -P(\la)$ holds for $\cM_A[-]$. Moreover, notice that, in the case $Q(\la) = -P(\la)$, part {\rm (a)} in the statement also holds for $P(\la)$, since any strong $\ell$-ification of $-P(\la)$ is also a strong $\ell$-ification of $P(\la)$. By parts {\rm (b2)}, and {\rm (b3)} in Theorem \ref{SLification}, part {\rm (b)} in the statement follows, since the left and right minimal indices of $P(\la)$ and $-P(\la)$ are the same.
	
	To prove {\rm (a)}, we only consider the case $A=A_2$, since the proof for the other cases is similar. This transformation appears only in parts {\rm (i)} (for $\ell$ even) and {\rm (ii)} (for $\ell$ odd) in the statement.
	
	We consider separately the M\"obius transformations $\cM_{A_2}[+]$ and $\cM_{A_2}[-]$. Let us start with $\cM_{A_2}[+]$. Looking at Table \ref{tableMa}, $\cM_{A_2}[L_d](\la^\ell) \otimes I_n = L_d(-\la^\ell) \otimes I_n$, if $\ell$ is odd, or $\cM_{A_2}[L_d](\la^\ell) \otimes I_n = L_d(\la^\ell) \otimes I_n$, if $\ell$ is even. As for $\cM_{A_2}[\Lambda_d](\la^\ell) \otimes I_n$, we have the following:
	\begin{itemize}
		\item $\cM_{A_2}[\Lambda_d](\la^\ell) \otimes I_n = \Lambda_d(\la^\ell) \otimes I_n$, if $\ell$ is even. Then, no block entry has a minus sign.
		\item $\cM_{A_2}[\Lambda_d](\la^\ell) \otimes I_n = \Lambda_d(-\la^\ell) \otimes I_n = \begin{bmatrix} -\la^{d\ell} & \cdots & +\la^{2\ell} & -\la^\ell & +1 \end{bmatrix}\otimes I_n$, if $\ell$ and $d$ are odd. Then, the block entries in odd positions of $\Lambda_d(-\la^\ell) \otimes I_n$ have a minus sign.
		\item $\cM_{A_2}[\Lambda_d](\la^\ell) \otimes I_n = \Lambda_d(-\la^\ell) \otimes I_n = \begin{bmatrix} +\la^{d\ell} & -\la^{(d-1)\ell}& \cdots & +\la^{2\ell} & -\la^\ell & +1 \end{bmatrix}\otimes I_n$, if $\ell$ is odd, and $d$ even. Then, the block entries in even positions of $\Lambda_d(-\la^\ell) \otimes I_n$ have a minus sign.
	\end{itemize}
	
	Therefore, if $\ell$ is odd, for any value of $d$,	we get
	$$
	\cM_{A_2}[\Lambda_d](\la^\ell) \otimes I_n = \begin{bmatrix} (-1)^d \la^{d\ell} & (-1)^{d-1}\la^{(d-1)\ell} & (-1)^{d-2}\la^{(d-2)\ell} & \cdots & +\la^{2\ell} & -\la^\ell & +1 \end{bmatrix}\otimes I_n.
	$$ 
	
	Now, we compute the product $Q(\la) = (\cM_{A_2}[\Lambda_d]^\top(\la^\ell) \otimes I_n)M(\la)(\Lambda_d(\la^\ell) \otimes I_n)$ using the previous expressions. Let us start assuming that $\ell$ is even. Then,
	\begin{equation*}
	\begin{array}{ccl}
	Q(\la) &= & [M]_{d+1,d+1} + \la^\ell \left([M]_{d,d+1} + [M]_{d+1,d} \right) + \la^{2\ell} \left([M]_{d-1,d+1} + [M]_{d,d} + [M]_{d+1,d-1}\right) + \cdots + \\
	&& \la^{d\ell} \left([M]_{1,d+1} + [M]_{2,d} + [M]_{3,d-1} + \cdots + [M]_{d-1,3} + [M]_{d,2} + [M]_{d+1,1}\right) + \cdots + \\
	&& \la^{(2d-2)\ell} \left( [M]_{1,3} + [M]_{2,2} + [M]_{3,1} \right) + \la^{(2d-1)\ell} \left( [M]_{1,2} + [M]_{2,1} \right) + \la^{2d\ell} [M]_{1,1},
	\end{array}
	\end{equation*}
	where $[M]_{s,t}$ is the grade-$\ell$ polynomial in the $(s,t)$ block of $M(\la)$. Then, since $M(\la)$ satisfies the {\rm AS} condition (part {\rm (AS)} in Definition \ref{structures}), the following identities are satisfied for each of the coefficients $P_j$, for $j=0,1,\hdots,k$:
	\begin{equation*}
	\begin{array}{rcl}
	P_0 &=& [M_0]_{d+1,d+1}, \\
	P_c &=& [M_c]_{d+1,d+1},\ {\rm for}\ c=1,2,\hdots,\ell-1,\\
	P_\ell &=& [M_\ell]_{d+1,d+1} + \left( [M_0]_{d,d+1} + [M_0]_{d+1,d}\right),\\
	P_{\ell +c} &=& [M_c]_{d,d+1} + [M_c]_{d+1,d},\ {\rm for}\ c=1,2,\hdots,\ell-1, \\
	P_{2\ell} &=& [M_\ell]_{d,d+1} + [M_\ell]_{d+1,d} + \left([M_0]_{d-1,d+1} + [M_0]_{d,d} + [M_0]_{d+1,d-1}\right), \\
	& \vdots \\
	P_{2d\ell} &=& \left([M_\ell]_{1,2} + [M_\ell]_{2,1}\right) + [M_0]_{1,1},\\
	P_{2d\ell + c} &=& [M_c]_{1,1},\ {\rm for}\ c=1,2,\hdots,\ell-1,\\
	P_{(2d+1)\ell} &=& [M_\ell]_{1,1}.
	\end{array}
	\end{equation*}
	Then, comparing the expressions for $Q(\lambda)$ and those for the coefficients $P_j$, for $j=0,1,\hdots, \ell(2d+1)=k$, and taking into account that $M(\lambda)=M_0+\lambda M_1+\cdots+\lambda^\ell M_\ell$, the identity $Q(\la) = P(\la)$ follows. Now, assume that $\ell$ is odd. Then,
	\begin{equation*}
	\begin{split}
	Q(\la) = & [M]_{d+1,d+1} + \la^\ell \left(-[M]_{d,d+1} + [M]_{d+1,d} \right) + \la^{2\ell} \left([M]_{d-1,d+1} - [M]_{d,d} + [M]_{d+1,d-1}\right) + \cdots+ \\
	& \la^{d\ell} \left((-1)^d [M]_{1,d+1} + (-1)^{d-1} [M]_{2,d} + (-1)^{d-2} [M]_{3,d-1} + \cdots + [M]_{d-1,3} - [M]_{d,2} + [M]_{d+1,1}\right) +  \\
	& \cdots +\la^{(2d-2)\ell} \left( (-1)^{d} [M]_{1,3} + (-1)^{d-1} [M]_{2,2} + (-1)^{d-2} [M]_{3,1} \right) +  \\
	& \la^{(2d-1)\ell} \left( (-1)^{d} [M]_{1,2} + (-1)^{d-1} [M]_{2,1} \right)+\la^{2d\ell} \left( (-1)^{d} [M]_{1,1}\right),
	\end{split}
	\end{equation*}
	so there are some signs that depend on the parity of $d$. Then, since $M(\la)$ satisfies the {\rm ASS} condition (part {\rm (ASS)} in Definition \ref{structures}), the coefficients $P_j$ satisfy, in this case, the following identities:
	\begin{equation*}
	\begin{split}
	P_0 =& [M_0]_{d+1,d+1}, \\
	P_c =& [M_c]_{d+1,d+1},\ {\rm for}\ c=1,2,\hdots,\ell-1,\\
	P_\ell =& [M_\ell]_{d+1,d+1} + \left( -[M_0]_{d,d+1} + [M_0]_{d+1,d}\right),\\
	P_{\ell +c} =& -[M_c]_{d,d+1} + [M_c]_{d+1,d},\ {\rm for}\ c=1,2,\hdots,\ell-1, \\
	P_{2\ell} =& -[M_\ell]_{d,d+1} + [M_\ell]_{d+1,d} + \left([M_0]_{d-1,d+1} - [M_0]_{d,d} + [M_0]_{d+1,d-1}\right), \\
	& \vdots \\
	P_{2d\ell} =& \left((-1)^{d} [M_\ell]_{1,2} + (-1)^{d-1} [M_\ell]_{2,1}\right) + (-1)^{d} [M_0]_{1,1},\\
	P_{2d\ell + c} =& (-1)^{d}[M_c]_{1,1},\ {\rm for}\ c=1,2,\hdots,\ell-1,\\
	P_{(2d+1)\ell} =& (-1)^{d}[M_\ell]_{1,1}.
	\end{split}
	\end{equation*}
	Notice that the only difference between these identities and those for the {\rm AS} condition (listed above for the case $\ell$ even), is the appearance of some minus signs in certain blocks of the polynomial $M(\la)$, depending on the respective rows they are placed in, together with the parity of $d$. Comparing, again, the coefficients of $Q(\la)$ and the expressions for $P_j$, for $j=0,1,\hdots,\ell(2d+1)=k$, and taking into account that $M(\la)=M_0+\la M_1+\cdots+\la^\ell M_\ell$, the identity $Q(\la) = P(\la)$ also holds in this case.
	
	Now, we consider the case $\cM_{A_2}[-]$. By Proposition \ref{PropersMaPs}-(c), $\cM_{A_2}[-L_d](\la^\ell) \otimes I_n = -\cM_{A_2}[L_d](\la^\ell) \otimes I_n$, and $\cM_{A_2}[-\Lambda_d](\la^\ell) \otimes I_n = -\cM_{A_2}[\Lambda_d](\la^\ell) \otimes I_n$. Reasoning as in the previous case, when we compute $Q(\la) = (\cM_{A_2}[-\Lambda_d]^\top(\la^\ell) \otimes I_n)M(\la)(\Lambda_d(\la^\ell) \otimes I_n) = -(\cM_{A_2}[\Lambda_d]^\top(\la^\ell) \otimes I_n)M(\la)(\Lambda_d(\la^\ell) \otimes I_n)$, since $M(\la)$ satisfies either the {\rm AS} condition, when $\ell$ is even, or the {\rm ASS} condition, when $\ell$ is odd, we conclude that $Q(\la) = -P(\la)$. Then the result is true for all appearances of $A=A_2$ in the statement.
\end{proof}

Notice that, in order for the matrix polynomial $\mathcal{L}(\la)$ as in \eqref{NoStructLlambda} to be an $\cM_A$-structured block-Kronecker degree-$\ell$ matrix polynomial (Definition \ref{StructBK}), the additional identity $\cM_A[\pm M] = M^\star(\la)$ must be satisfied, that is, $M(\la)$ must be $\cM_A$-structured. Theorem \ref{Ma-strucstrong} provides a way to construct a structure-preserving strong $\ell$-ification of a given $\cM_A$-structured matrix polynomial of grade $k$ via an $\cM_A$-structured strong block minimal bases grade-$\ell$ matrix polynomial with a particular $(1,1)$ big block (like the one in \eqref{MaL1/2}). More precisely, and focusing on the block-Kronecker dual minimal bases, what we need is a solution of \eqref{StructPlambdaKronecker}, $M(\la)$, not necessarily $\cM_A$-structured, and then place the polynomial $\frac{1}{2}\left(M + \cM_A[\pm M]^\star\right)(\la)$ in the $(1,1)$ big block of the $\cM_A$-structured block-Kronecker degree-$\ell$ matrix polynomial as in \eqref{StructLSBlockKronecker}. But Theorem \ref{NoStructLlambdaTeo} provides a way to construct a solution of \eqref{StructPlambdaKronecker} depending on the matrix $A$ associated with any of the structures in Definition \ref{Structures-Mobius}. Namely, we just need to follow either the {\rm AS}, {\rm ASS}, or {\rm DS} conditions, depending of the structure and the parity of $\ell$. These considerations lead us to the following result. 

\begin{Teo}\label{StructLgeneral}
	Let $P(\la) = \sum_{j=0}^k \la^j P_j \in \FF[\la]^{n \times n}$ be an $\cM_A$-structured matrix polynomial of grade $k$, with $A$ being any of the matrices in Definition {\rm\ref{Structures-Mobius}}. Let $M(\la) = \sum_{i=0}^\ell \la^i M_i \in \FF[\la]^{(d+1)n \times (d+1)n}$ be a grade-$\ell$ matrix polynomial, with $k=(2d+1)\ell$, for some $d\geq 0$. Let us partition the matrices $M_i$, for $i=0,1,\hdots,\ell$, into $(d+1) \times (d+1)$ blocks, each of size $n \times n$, and let us denote the blocks by $[M_i]_{s,t} \in \FF^{n \times n}$ for $s,t = 1,2,\hdots,d+1$, and $i=0,1,\hdots,\ell$. Then, if $M(\la)$ satisfies either {\rm (i)}, {\rm (ii)}, or {\rm (iii)} in Theorem {\rm\ref{NoStructLlambdaTeo}}, the matrix polynomial
	\begin{equation}\label{GeneralLlambda}
	\mathcal{L}(\la) = \begin{bmatrix} \frac{1}{2}\left(M + \cM_A[\pm M]^\star\right)(\la) & \cM_A[\pm L_d]^\top(\la^\ell) \otimes I_n \\ L_d(\la^\ell) \otimes I_n & 0 \end{bmatrix}
	\end{equation}
	is an $\cM_A$-structured block-Kronecker degree-$\ell$ matrix polynomial such that:
	\begin{enumerate}[{\rm (a)}]
		\item $\mathcal{L}(\la)$ is a strong $\ell$-ification of $P(\la)$,
		\item if $0 \le \epsilon_1 \le \epsilon_2 \le \cdots \le \epsilon_p$ are the right minimal indices of $P(\la)$, then $\epsilon_1 + d\ell \le \epsilon_2 + d\ell \le \cdots \le \epsilon_p + d\ell$ are the right minimal indices of $\mathcal{L}(\la)$, and
		\item if $0 \le \eta_1 \le \eta_2 \le \cdots \le \eta_p$ are the left minimal indices of $P(\la)$, then $\eta_1 + d\ell \le \eta_2 + d\ell \le \cdots \le \eta_p + d\ell$ are the left minimal indices of $\mathcal{L}(\la)$.
	\end{enumerate}
\end{Teo}

\begin{Rem}
{\rm
	In the statement of Theorem \ref{StructLgeneral}, the $\pm$ signs in \eqref{GeneralLlambda} must be taken in accordance with the appropriate identity in \eqref{Ma-struct} for the $\cM_A$-structured matrix polynomial $P(\la)$. 
}
\end{Rem}

\begin{proof}[Proof (of Theorem {\rm\ref{StructLgeneral}})] By Proposition \ref{PropersMaPs}, and using that $A$ is coninvolutory, it is straightforward to check that $\mathcal{L}(\la)$ is $\cM_A$-structured. In particular, it is an $\cM_A$-structured block-Kronecker degree-$\ell$ matrix polynomial (Definition \ref{StructBK}). To prove parts {\rm (a)--(c)}, we need to check that \eqref{StructPlambdaKronecker} holds for the matrix polynomial $\frac{1}{2}\left(M + \cM_A[\pm M]^\star\right)(\la)$ (up to a sign, specifically in the cases associated to the M\"obius transformation $\cM_A[-]$), and then, the results follow immediately from Theorem \ref{StructPLambdaKroneckerTeo}.
	
	We focus only on $A=A_1$, since the proof for the other cases is similar. In particular, $M(\la)$ satisfies the {\rm AS} condition for $P(\la)$ (see Theorem \ref{NoStructLlambdaTeo}). 
	
	First, we consider the case $\cM_{A_1}[+]$. Since $A_1 = I_2$, we have $\cM_{A_1}[L_d](\la^\ell) \otimes I_n = L_d(\la^\ell) \otimes I_n$, $\cM_{A_1}[\Lambda_d](\la^\ell) \otimes I_n = \Lambda_d(\la^\ell) \otimes I_n$, and $\cM_{A_1}[M] = M(\la)$. By Theorem \ref{StructPLambdaKroneckerTeo}, the matrix polynomial
	\begin{equation*}
	\mathcal{L}(\la) = \begin{bmatrix} \frac{1}{2}\left(M + \cM_{A_1}[M]^\star\right)(\la) & \cM_{A_1}[L_d]^\top(\la^\ell) \otimes I_n \\ L_d(\la^\ell) \otimes I_n & 0 \end{bmatrix} = \begin{bmatrix} \frac{1}{2}\left(M + M^\star\right)(\la) & L_d^\top(\la^\ell) \otimes I_n \\ L_d(\la^\ell) \otimes I_n & 0 \end{bmatrix}
	\end{equation*}
	is a strong $\ell$-ification of 
	\begin{equation*}
	P_{{\rm A}_1^+}(\la) := (\Lambda_d^\top(\la^\ell) \otimes I_n)\left(\frac{1}{2}\left(M + M^\star\right)(\la)\right)(\Lambda_d(\la^\ell) \otimes I_n).
	\end{equation*}
	If we compute the product of these three blocks, using $\widetilde{M} := \frac{1}{2}\left(M + M^\star\right)$ to abbreviate, we obtain the following polynomial:
	\begin{equation*}
	\begin{split}
	P_{{\rm A}_1^+}(\la) = & [\widetilde{M}]_{d+1,d+1} + \la^\ell \left([\widetilde{M}]_{d,d+1} + [\widetilde{M}]_{d+1,d} \right) + \la^{2\ell} \left([\widetilde{M}]_{d-1,d+1} + [\widetilde{M}]_{d,d} + [\widetilde{M}]_{d+1,d-1}\right) + \cdots + \\
	& \la^{d\ell} \left([\widetilde{M}]_{1,d+1} + [\widetilde{M}]_{2,d} + [\widetilde{M}]_{3,d-1} + \cdots + [\widetilde{M}]_{d-1,3} + [\widetilde{M}]_{d,2} + [\widetilde{M}]_{d+1,1}\right) + \cdots + \\
	& \la^{(2d-2)\ell} \left( [\widetilde{M}]_{1,3} + [\widetilde{M}]_{2,2} + [\widetilde{M}]_{3,1} \right) + \la^{(2d-1)\ell} \left( [\widetilde{M}]_{1,2} + [\widetilde{M}]_{2,1} \right) + \la^{2d\ell} [\widetilde{M}]_{1,1},
	\end{split}
	\end{equation*}
	where $[\widetilde{M}]_{s,t}$ is the grade-$\ell$ polynomial in the $(s,t)$ block of $\widetilde{M}(\la)$. Since $P(\la)$ is $\cM_{A_1}$-structured, in the sense that $\cM_{A_1}[P] = P^\star(\la)$ (i.e., $P_j = P_j^\star$, for $j=0,1,\hdots,k$), and $M(\la)$ satisfies the {\rm AS} condition (part {\rm (AS)} in Definition \ref{structures}), we conclude that $P_{{\rm A}_1^+}(\la) = \frac{1}{2}\left(P + P^\star\right)(\la) = P(\la)$. 
	
	Now, we consider the case $\cM_{A_1}[-]$. Then, $\cM_{A_1}[-L_d](\la^\ell) \otimes I_n = -L_d(\la^\ell) \otimes I_n$, $\cM_{A_1}[-\Lambda_d](\la^\ell) \otimes I_n = -\Lambda_d(\la^\ell) \otimes I_n$, and $\cM_{A_1}[-M] = -M(\la)$. By Theorem \ref{StructPLambdaKroneckerTeo}, the matrix polynomial
	\begin{equation*}
	\mathcal{L}(\la) = \begin{bmatrix} \frac{1}{2}\left(M + \cM_{A_1}[-M]^\star\right)(\la) & \cM_{A_1}[-L_d]^\top(\la^\ell) \otimes I_n \\ L_d(\la^\ell) \otimes I_n & 0 \end{bmatrix} = \begin{bmatrix} \frac{1}{2}\left(M - M^\star\right)(\la) & -L_d^\top(\la^\ell) \otimes I_n \\ L_d(\la^\ell) \otimes I_n & 0 \end{bmatrix}
	\end{equation*}
	is a strong $\ell$-ification of 
	\begin{equation*}
	\begin{array}{rcl}
	P_{{\rm A}_1^-}(\la) &:=& (-\Lambda_d^\top(\la^\ell) \otimes I_n)\left(\frac{1}{2}\left(M - M^\star\right)(\la)\right)(\Lambda_d(\la^\ell) \otimes I_n)  \\
	&=& -(\Lambda_d^\top(\la^\ell) \otimes I_n)\left(\frac{1}{2}\left(M - M^\star\right)(\la)\right)(\Lambda_d(\la^\ell) \otimes I_n).
	\end{array}
	\end{equation*}
	Reasoning as in the previous case, since $M(\la)$ satisfies the {\rm AS} condition for $P(\la)$ and $P(\la)$ is $\cM_{A_1}$-structured, in the sense that $\cM_{A_1}[-P] = P^\star(\la)$ (which implies $-P_j = P_j^\star$, for $j=0,1,\hdots,k$), we conclude that $P_{{\rm A}_1^-}(\la) =-\frac{1}{2}(P-P^\star)(\la)= -P(\la)$. Therefore, parts (a)--(c) in the statement are satisfied, since any strong $\ell$-ification of $-P(\la)$ is also a strong $\ell$-ification of $P(\la)$, and the sets of left and right minimal indices of $P(\la)$ and $-P(\la)$ are the same.
\end{proof}

Theorem \ref{StructLgeneral} not only provides a way to construct the block $M(\la)$ in \eqref{StructLSBlockKronecker} to obtain a structure-preserving strong $\ell$-ification $\mathcal{L}(\la)$ of an $\cM_A$-structured matrix polynomial for all the structures in Definition \ref{Structures-Mobius}, but also presents a wide family of structured strong $\ell$-ifications. This family is wide in the sense that the polynomial $M(\la)$ can be constructed, in general, in many different ways just following the requeriments provided by either the {\rm AS}, {\rm ASS}, or {\rm DS} conditions. These conditions impose some restrictions on the (anti-) diagonals of $M(\la)$ in terms of the coefficients of $P(\la)$, but allow for some flexibility, and it is quite easy to construct different polynomials $M(\la)$ just following these rules. Once $M(\la)$ is built, we place the matrix polynomial $\frac{1}{2}\left(M + \cM_A[\pm M]^\star\right)(\la)$ in the $(1,1)$ big block of \eqref{GeneralLlambda} to get a structure-preserving strong $\ell$-ification of $P(\la)$.

Table \ref{tableRelation} summarizes the relevant information in Theorem \ref{StructLgeneral} (conditions {\rm (i)}--{\rm (iii)}), and relates it with the particular structures introduced in Definition \ref{Structures-Mobius}.

\begin{table}[ht]
	\centering
	\begin{tabular}{| c | c | c || c |}
		\hline
		Conditions for $P(\la)$ & M\"obius transformation $\cM_A[\pm]$ & Parity of $\ell$ & Structure\\
		\hline
		\hline
		\multirow{4}{0.7cm}{AS} & $\cM_{A_1}[+]$ & any $\ell$ & $\star$-symmetric \\ \cline{2-4} & $\cM_{A_1}[-]$ & any $\ell$ & $\star$-skew-symmetric \\ \cline{2-4} & $\cM_{A_2}[+]$ & $\ell$ even & $\star$-even \\ \cline{2-4} & $\cM_{A_2}[-]$ & $\ell$ even & $\star$-odd \\
		\hline
		\multirow{2}{0.7cm}{ASS} & $\cM_{A_2}[+]$ & $\ell$ odd  & $\star$-even \\ \cline{2-4} & $\cM_{A_2}[-]$ & $\ell$ odd & $\star$-odd \\
		\hline
		\multirow{2}{0.7cm}{DS} & $\cM_{A_3}[+]$  & any $\ell$ & $\star$-palindromic \\ \cline{2-4} & $\cM_{A_3}[-]$ & any $\ell$ & $\star$-anti-palindromic \\
		\hline
	\end{tabular}
	\caption{Relation between the conditions and the structures associated to its M\"obius transformation.}
	\label{tableRelation}
\end{table}

In the following example we illustrate Theorem \ref{StructLgeneral}, by explicitly constructing structure-preserving strong quadratifications for $\star$-symmetric, $\star$-odd, and $\star$-palindromic matrix polynomials of grade $10$.
\begin{Exa}\label{ExamDuplic}
{\rm
	Let $P(\la) = \sum_{j=0}^{10} \la^j P_j \in \FF[\la]^{n \times n}$ be a matrix polynomial of grade $10$. Let us consider the following three quadratic matrix polynomials:
	\begin{equation*}
	M_1(\la) := {\footnotesize \left[\begin{array}{ccc} \la P_9 + \la^2 P_{10} & P_6 + \la P_7 + \la^2 P_8  & 0 \\ 0 & P_4 + \la P_5 & P_2 + \la P_3 \\ 0 & 0 & P_0 + \la P_1 \end{array}\right]},
	\end{equation*}
	\begin{equation*}
	\begin{split}
	M_2(\la) &:= {\footnotesize\left[\begin{array}{ccc} P_8 + \la P_9 + \la^2 P_{10} & 0 & P_4 \\ P_6 + \la P_7 & \la P_5 & 0 \\ 0 & \la P_3 & P_0 + \la P_1 +\la^2 P_2 \end{array}\right]},\ {\rm and} \\
	M_3(\la) &:= {\footnotesize\left[\begin{array}{ccc} P_4 + \la P_5 + \la^2 P_6 & 0 & P_0 + \la P_1 + \la^2 P_2 \\ \la P_7 & 0 & \la P_3 \\ P_8 + \la P_9 + \la^2 P_{10} & 0 & 0 \end{array}\right]}.
	\end{split}
	\end{equation*}
	It is straightforward to check that $M_1(\la)$ and $M_2(\la)$ satisfy the {\rm AS} condition, and that $M_3(\la)$ satisfies the {\rm DS} condition for $P(\la)$.
	
	If $A_1$ is the matrix in Definition {\rm\ref{Structures-Mobius}}, and $P(\la)$ is $\star$-symmetric, the quadratic matrix polynomial\break $\frac{1}{2}\left(M_1 + \cM_{A_1}[M_1]^\star\right)(\la)$ is equal to
	\begin{equation*}
	\frac{1}{2}\left(M_1 + M_1^\star\right)(\la) = {\footnotesize \left[\begin{array}{ccc} \la P_9 + \la^2 P_{10} & (P_6 + \la P_7 + \la^2 P_8)/2  & 0 \\ (P_6 + \la P_7 + \la^2 P_8)/2 & P_4 + \la P_5 & (P_2 + \la P_3)/2 \\ 0 & (P_2 + \la P_3)/2 & P_0 + \la P_1 \end{array}\right]},
	\end{equation*}
	which is a $\star$-symmetric quadratic matrix polynomial and, by Theorem {\rm\ref{StructLgeneral}}, the $\star$-symmetric block-Kronecker quadratic matrix polynomial
	\begin{equation*}
	\begin{array}{ccl}
	\mathcal{L}_{\mathcal{S}}(\la) &:=& \left[\begin{array}{c|c} \frac{1}{2}\left(M_1 + M_1^\star\right)(\la) & L_2^\top(\la^2) \otimes I_n \\ \hline L_2(\la^2) \otimes I_n & 0 \end{array}\right] \\
	&=& {\footnotesize \left[\begin{array}{c|c} \begin{array}{ccc} \la P_9 + \la^2 P_{10} & (P_6 + \la P_7 + \la^2 P_8)/2  & 0 \\ (P_6 + \la P_7 + \la^2 P_8)/2 & P_4 + \la P_5 & (P_2 + \la P_3)/2 \\ 0 & (P_2 + \la P_3)/2 & P_0 + \la P_1 \end{array} & \begin{array}{cc} -I_n & 0 \\ \la^2 I_n & -I_n \\ 0 & \la^2 I_n \end{array} \\ \hline \begin{array}{c@{\mskip 150mu}c@{\mskip 120mu}c@{\mskip -20mu}} -I_n & \la^2 I_n & 0 \\ 0  & -I_n & \la^2 I_n \end{array} & \begin{array}{c@{\mskip 45mu}c} 0 & 0 \\ 0 & 0 \end{array} \end{array}\right]}
	\end{array}
	\end{equation*}
	is a $\star$-symmetric strong quadratification of $P(\la)$ whenever $P(\la)$ is a $\star$-symmetric matrix polynomial. 
	
	If $A_2$ is the matrix in Definition {\rm\ref{Structures-Mobius}}, and $P(\la)$ is $\star$-odd, the quadratic matrix polynomial \\
	$\frac{1}{2}\left(M_2 +\cM_{A_2}[M_2]^\star\right)(\la)$ is equal to
	\begin{equation*}
	\frac{1}{2}\left(M_2(\la) - M_2^\star(-\la)\right) = {\footnotesize \left[\begin{array}{ccc} P_8 + \la P_9 + \la^2 P_{10} & (P_6 + \la P_7)/2 & P_4/2 \\ (P_6 + \la P_7)/2 & \la P_5 & \la P_3/2 \\ P_4/2 & \la P_3/2 & P_0 + \la P_1 +\la^2 P_2 \end{array}\right]},
	\end{equation*}
	which is a $\star$-odd quadratic matrix polynomial and, by Theorem {\rm\ref{StructLgeneral}}, the $\star$-odd block-Kronecker quadratic matrix polynomial
	\begin{equation*}
	\begin{array}{ccl}
	\mathcal{L}_{\mathcal{O}}(\la) &:=& \left[\begin{array}{c|c} \frac{1}{2}\left(M_2(\la) - M_2^\star(-\la)\right) & -L_2^\top(\la^2) \otimes I_n \\ \hline L_2(\la^2) \otimes I_n & 0 \end{array}\right] \\
	&= &{\footnotesize \left[\begin{array}{c|c} \begin{array}{ccc} P_8 + \la P_9 + \la^2 P_{10} & (P_6 + \la P_7)/2 & P_4/2 \\ (P_6 + \la P_7)/2 & \la P_5 & \la P_3/2 \\ P_4/2 & \la P_3/2 & P_0 + \la P_1 +\la^2 P_2 \end{array} & \begin{array}{cc} I_n & 0 \\ -\la^2 I_n & I_n \\ 0 & -\la^2 I_n \end{array} \\ \hline \begin{array}{c@{\mskip 115mu}c@{\mskip 100mu}c}  -I_n & \la^2 I_n & 0 \\ 0 & -I_n & \la^2 I_n \end{array} & \begin{array}{c@{\mskip 60mu}c} 0 & 0 \\ 0 & 0 \end{array} \end{array}\right]}
	\end{array}
	\end{equation*}
	is a $\star$-odd strong quadratification of $P(\la)$ (although, in this case, the following identity is the right one $(-\Lambda_2^\top(\la^2) \otimes I_n)\left(\frac{1}{2}\left(M_2(\la) - M_2^\star(-\la)\right)\right)(\Lambda_2(\la^2) \otimes I_n) := -P(\la)$) whenever $P(\la)$ is a $\star$-odd matrix polynomial. 
	
	Finally, if $A_3$ is the matrix in Definition {\rm\ref{Structures-Mobius}}, and $P(\la)$ is $\star$-palindromic, the quadratic matrix polynomial $\frac{1}{2}\left(M_3 + \cM_{A_3}[M_3]^\star\right)(\la)$ is equal to
	\begin{equation*}
	\frac{1}{2}\left(M_3 + {\rm rev}_2M_3^\star\right)(\la) = {\footnotesize \left[\begin{array}{ccc} P_4 + \la P_5 + \la^2 P_6 & \la P_3/2 & P_0 + \la P_1 + \la^2 P_2 \\ \la P_7/2 & 0 & \la P_3/2 \\ P_8 + \la P_9 + \la^2 P_{10} & \la P_7/2 & 0 \end{array}\right]},
	\end{equation*}
	which is a $\star$-palindromic quadratic matrix polynomial and, by Theorem {\rm\ref{StructLgeneral}}, the $\star$-palindromic block-Kronecker quadratic matrix polynomial
	\begin{equation*}
	\begin{array}{ccl}
	\mathcal{L}_{\mathcal{P}}(\la) &:= &\left[\begin{array}{c|c} \frac{1}{2}\left(M_3 + {\rm rev}_2M_3^\star\right)(\la) & {\rm rev}_2(L_2)^\top(\la^2) \otimes I_n \\ \hline L_2(\la^2) \otimes I_n & 0 \end{array}\right] \\
	&= &{\footnotesize \left[\begin{array}{c|c} \begin{array}{ccc} P_4 + \la P_5 + \la^2 P_6 & \la P_3/2 & P_0 + \la P_1 + \la^2 P_2 \\ \la P_7/2 & 0 & \la P_3/2 \\ P_8 + \la P_9 + \la^2 P_{10} & \la P_7/2 & 0 \end{array} & \begin{array}{cc} -\la^2 I_n & 0 \\ I_n & -\la^2 I_n \\ 0 & I_n \end{array} \\ \hline \begin{array}{c@{\mskip 80mu}c@{\mskip 75mu}c} -I_n & \la^2 I_n & 0 \\ 0 & -I_n & \la^2 I_n \end{array} & \begin{array}{c@{\mskip 60mu}c} 0 & 0 \\ 0 & 0 \end{array} \end{array}\right]}
	\end{array}
	\end{equation*}
	is a $\star$-palindromic strong quadratification of $P(\la)$ whenever $P(\la)$ is a $\star$-palindromic matrix polynomial. 
	
	Notice also that the matrix polynomials 
	\begin{equation*}
	\mathcal{L}_i(\la) = \left[\begin{array}{c|c} M_i(\la) & \cM_{A_i}[\pm L_2]^\top(\la^2) \otimes I_n \\ \hline L_2(\la^2) \otimes I_n & 0 \end{array}\right],\ {\rm for}\ i=1,2,3,
	\end{equation*}
	are, by Theorem {\rm\ref{NoStructLlambdaTeo}}, strong quadratifications of $P(\la)$, even if $P(\la)$ is not structured. 
	
	The previous constructions can be considered, of course, even if the original polynomial $P(\la)$ is not structured. However, in this case they do not necessarily provide a strong $\ell$-ification of the original polynomial $P(\la)$. Let us focus, for instance, on the matrix polynomial $M_1(\la)$, and assume that $P(\la)=\sum_{j=0}^k \la^j P_j$ is a general (not necessarily $\star$-symmetric) matrix polynomial. The quadratic matrix polynomial $\frac{1}{2}\left(M_1 + \cM_{A_1}[M_1]^\star\right)(\la)$ is equal to
	\begin{equation*}
	\frac{1}{2}\left(M_1 + M_1^\star\right)(\la) = {\footnotesize \left[\begin{array}{ccc} \la (P_9+P_9^\star) + \la^2 (P_{10}+P_{10}^\star) & (P_6 + \la P_7 + \la^2 P_8)/2  & 0 \\ (P_6^\star + \la P_7^\star + \la^2 P_8^\star)/2 & (P_4+P_4^\star) + \la (P_5+P_5^\star) & (P_2 + \la P_3)/2 \\ 0 & (P_2^\star + \la P_3^\star)/2 & (P_0+P_0^\star) + \la (P_1+P_1^\star) \end{array}\right]},
	\end{equation*}
	which is a $\star$-symmetric quadratic matrix polynomial and, by Theorem {\rm\ref{Ma-strucstrong}}, the $\star$-symmetric block-Kronecker quadratic matrix polynomial
	\begin{equation*}
	\begin{array}{ccl}
	{\footnotesize \left[\begin{array}{c|c} \begin{array}{ccc} \la (P_9+P_9^\star)/2 + \la^2 (P_{10}+P_{10}^\star)/2 & (P_6 + \la P_7 + \la^2 P_8)/2  & 0 \\ (P_6^\star + \la P_7^\star + \la^2 P_8^\star)/2 & (P_4+P_4^\star)/2 + \la (P_5+P_5^\star)/2 & (P_2 + \la P_3)/2 \\ 0 & (P_2^\star + \la P_3^\star)/2 & (P_0+P_0^\star)/2 + \la (P_1+P_1^\star)/2 \end{array} & \begin{array}{cc} -I_n & 0 \\ \la^2 I_n & -I_n \\ 0 & \la^2 I_n \end{array} \\ \hline \begin{array}{c@{\mskip 250mu}c@{\mskip 200mu}c@{\mskip -15mu}} -I_n & \la^2 I_n & 0 \\ 0  & -I_n & \la^2 I_n \end{array} & \begin{array}{c@{\mskip 50mu}c} 0 & 0 \\ 0 & 0 \end{array} \end{array}\right]}
	\end{array}
	\end{equation*}
	is a $\star$-symmetric strong quadratification of the matrix polynomial
	$$Q(\la) := (\Lambda_2^\top(\la^2) \otimes I_n)\left(\frac{1}{2}\left(M_1 + M_1^\star\right)\right)(\Lambda_2(\la^2) \otimes I_n) = \sum_{j=0}^k \la^j (P_j + P_j^\star),$$
	which is equal to the matrix polynomial $P(\la)$ if and only if $P(\la)$ is $\star$-symmetric.
}
\end{Exa}

\subsection{Sparse $\boldsymbol{\cM_A}$-structured block-Kronecker matrix polynomials} \label{SectionSparse}

In this section we analyze the sparsity of the strong $\ell$-ifications as in \eqref{GeneralLlambda}, and we will present a procedure to construct $\ell$-ifications in this family having exactly the minimum number of nonzero block entries. We will first determine, in Proposition \ref{PropSparse}, the smallest number of nonzero block entries of a matrix polynomial as in \eqref{NoStructLlambda} (not necessarily $\cM_A$-structured). To obtain this minimum number of nonzero block entries, we focus on the matrix polynomial $M(\la)$, placed in the $(1,1)$ big block of \eqref{NoStructLlambda}, because the two minimal bases, placed in the $(1,2)$ and $(2,1)$ big blocks, always have the same number of nonzero block entries.

%Then, we start by analyzing when the strong $\ell$-ification as in \eqref{GeneralLlambda} has the minimum number of nonzero block entries. In particular, we focus on the matrix polynomial $\frac{1}{2}\left(M + M_A[\pm M]^\star\right)(\la)$, because the two minimal bases in the $(1,2)$ and $(2,1)$ blocks always have the same number of nonzero block entries. 

\begin{Prop}\label{PropSparse}
	Let $P(\la)$, $M(\la)$, and $\mathcal{L}(\la)$ be as in the statement of Theorem {\rm \ref{NoStructLlambdaTeo}}. Then, $\mathcal{L}(\la)$ has, at least,
	\begin{enumerate}[{\rm (a)}]
		\item $2k-1+\left\lfloor \frac{k}{2} \right\rfloor=5d+1$ nonzero block entries, if $\ell=1$, or
		\item $\frac{3k}{\ell}-2=6d+1$ nonzero block entries, if $\ell>1$.
	\end{enumerate}
	%Moreover, for any $k$ and $\ell$ there is an $M_A$-structured block-Kronecker $\ell$-ification ${\cal L}(\la)$ having exactly this number of nonzero block entries.
\end{Prop}

\begin{proof}
	It is proved in \cite[Th. 52]{de2020gen} that any generalized companion pencil with, at least, $2(k-1)$ nonzero block entries in $\FF^{n \times n}$, has, at least, $2k-1+\left\lfloor \frac{k}{2} \right\rfloor$ nonzero block entries. Then, {\rm (a)} follows, because the pencils $L_d(\la) \otimes I_n$ and $\cM_A[\pm L_d]^\top(\la) \otimes I_n$ in \eqref{NoStructLlambda} have a total amount of $4d=2(k-1)$ nonzero blocks.

To prove {\rm (b)}, we are going to see that, for $\ell>1$, an ($\cM_A$-structured) block-Kronecker degree-$\ell$ matrix polynomial as in \eqref{NoStructLlambda} has $6d+1$ nonzero block entries. First, and as before, the polynomials $L_d(\la^\ell) \otimes I_n$ and $\cM_A[\pm L_d]^\top(\la^\ell) \otimes I_n$ in \eqref{NoStructLlambda} give us a total amount of $4d$ nonzero blocks. Second, the identity on the coefficients $P_{\ell r + c}$ in either the {\rm AS}, {\rm ASS}, or {\rm DS} condition for $P(\la)$ (second line of each of these conditions in Definition \ref{structures}), implies that all the $2d+1$ (anti-)diagonals of $M(\la)$ have, at least, a nonzero block entry containing coefficients of $P(\la)$. Therefore, $\mathcal{L}(\la)$ as in \eqref{NoStructLlambda} has, at least, $4d+2d+1 = 6d+1$ nonzero block entries. %Moreover, the coefficients can be grouped in just one nonzero block entry for each block (anti)diagonal because the {\rm AS}, {\rm ASS} and {\rm DS} conditions allow for flexibility in placing the coefficients of $P(\la)$ in the corresponding block (anti)diagonal. Although different groupings are possible, they all contain, at least, one nonzero entry in each (anti)diagonal. 
\end{proof}

%A sparse $\cM_A$-structured block-Kronecker degree-$\ell$ matrix polynomial as in \eqref{GeneralLlambda} must have the same number of nonzero block entries than a sparse block-Kronecker degree-$\ell$ matrix polynomial as in \eqref{NoStructLlambda}, which is not necessarily $\cM_A$-structured, since both of them have the same {\color{blue} block structure}. Therefore, since 

Proposition \ref{PropSparse} motivates the following definition for structure-preserving $\ell$-ifications.

\begin{Def}\label{DefSparse}
{\rm
	A {\em sparse $\cM_A$-structured block-Kronecker strong $\ell$-ification of $P(\la)$} is a matrix polynomial satisfying Theorem {\rm\ref{StructLgeneral}} with exactly
	\begin{enumerate}[{\rm (a)}]
		\item $2k-1+\left\lfloor \frac{k}{2} \right\rfloor=5d+1$ nonzero block entries, if $\ell=1$, or
		\item $\frac{3k}{\ell}-2=6d+1$ nonzero block entries, if $\ell>1$.
	\end{enumerate}
}
\end{Def}

Proposition \ref{PropSparse} provides the smallest number of nonzero block entries within the family of strong $\ell$-ifications in Theorem \ref{NoStructLlambdaTeo}, namely those of the form \eqref{NoStructLlambda}. We are interested, however, in structure-preserving strong $\ell$-ifications as in \eqref{GeneralLlambda}, namely those satisfying Theorem \ref{StructLgeneral} (like in Definition \ref{DefSparse}). A natural question after Proposition \ref{PropSparse} is whether or not there are sparse $\cM_A$-structured block-Kronecker strong $\ell$-ifications as in \eqref{GeneralLlambda}, for any $k,\ell$, and any of the structures in Definition \ref{Structures-Mobius}. The answer is positive for $\ell=1$, namely for linearizations.

In order to get a linearization ${\cal L}(\la)$ as in \eqref{GeneralLlambda} with exactly the number of nonzero block entries indicated in Definition \ref{DefSparse}-(a), we set
$$
M_1(\la)= \begin{bmatrix}P_{k-1}+\la P_k\\&\ddots\\&&P_0+\la P_1\end{bmatrix},
$$
for $A=A_1$, 
$$
M_2(\la)=\left\{\begin{array}{cc}\footnotesize\begin{bmatrix}P_{k-1}+\la P_k\\&-(P_{k-3}+\la P_{k-1})\\&&\ddots\\&&&-(P_2+\la P_3)\\&&&&P_0+\la P_1\end{bmatrix},&\mbox{if $d$ is even},\\
\footnotesize\begin{bmatrix}-(P_{k-1}+\la P_k)\\&P_{k-3}+\la P_{k-1}\\&&\ddots\\&&&-(P_2+\la P_3)\\&&&&P_0+\la P_1\end{bmatrix},&\mbox{if $d$ is odd},\end{array}\right.
$$
for $A=A_2$, and 
$$
M_3(\la)=\begin{bmatrix}&&P_0+\la P_1\\&\iddots&\\P_{k-1}+\la P_k&&\end{bmatrix}
$$
for $A=A_3$. 

The matrix polynomial $M_1(\la)$ satisfies the AS condition, the matrix polynomial $M_2(\la)$ satisfies the ASS condition, and the matrix polynomial $M_3(\la)$ satisfies the DS condition, and they all have exactly $d+1$ nonzero block entries. Therefore, the pencil $\frac{1}{2}(M_i+\cM_{A_i}[\pm M_i]^\star)$, for $i=1,2,3$, is an $\cM_{A_i}$-structured pencil, provided that $P(\la)$ is $\cM_{A_i}$-structured, and has also exactly $d+1$ nonzero block entries, so the corresponding pencil \eqref{GeneralLlambda} is an $\cM_{A_i}$-structured pencil having exactly $2k-2+d+1=5d+1$ nonzero block entries. 

Looking at the previous constructions, we see that the pencils $\frac{1}{2}(M_i+\cM_{A_i}[\pm M_i]^\star)(\la)$, for $i=1,2$, have all their nonzero block entries in the main diagonal. It is not difficult to see that, in order for ${\cal L}(\la)$ as in Theorem \ref{StructLgeneral}, with $\ell=1$, to be sparse, for $A=A_1,A_2$, all nonzero block entries must be on the main diagonal, and they have the same structure as $M_1(\la)$ and $M_2(\la)$, up to multiplication by nonzero constants. Note also that the pencil $\frac{1}{2}(M_3+\cM_{A_3}[\pm M_3]^\star)(\la)$ has all its nonzero block entries on the main anti-diagonal. However, this is not necessarily the case for all sparse matrix pencils in Theorem \ref{StructLgeneral} with $A=A_3$ (see Example \ref{ExampleNoDuplicSparse}). 

We want to emphasize that the previous constructions are not new, since they are particular cases of the family introduced in \cite{dopico2019structured}. However, for $\ell>1$, the situation is different, since, as we have seen in Proposition \ref{PropSparse}, every anti-diagonal of $M(\la)$ in \eqref{GeneralLlambda} must contain, at least, one nonzero block entry. Before proceeding with the answer for the case $\ell>1$, we present a procedure for constructing sparse $\ell$-ifications as in \eqref{GeneralLlambda}, following the proof of Proposition \ref{PropSparse}. 

\subsubsection*{Procedure to construct a sparse $\boldsymbol{\cM_A}$-structured block-Kronecker strong $\ell$-ification}

We aim to present a procedure for constructing sparse $\cM_A$-structured block-Kronecker strong $\ell$-ifications $\mathcal{L}(\la)$, introduced in Definition \ref{DefSparse}, for the different structures in Definition \ref{Structures-Mobius}. This procedure is stated below. In particular, we explain how to construct the matrix polynomial $M(\la)$ in order to guarantee that the $(1,1)$ big block $\frac{1}{2}(M + \cM_A[\pm M]^\star)(\la)$ of $\mathcal{L}(\la)$ as in \eqref{GeneralLlambda} has the minimum number of nonzero block entries.

Let us assume that $P(\la) = \sum_{j=0}^k \la^j P_j \in \FF[\la]^{n \times n}$ is an $\cM_A$-structured matrix polynomial of grade $k$, with $A$ being any of the matrices in Definition \ref{Structures-Mobius}. Let $M(\la) = \sum_{i=0}^\ell \la^i M_i \in \FF[\la]^{(d+1)n \times (d+1)n}$ be a grade-$\ell$ matrix polynomial, with $k=(2d+1)\ell$, for some $d\geq 0$. Let us partition the matrices $M_i$, for $i=0,1,\hdots,\ell$, into $(d+1) \times (d+1)$ blocks, each of size $n \times n$, and let us denote the blocks by $[M_i]_{s,t} \in \FF^{n \times n}$ for $s,t = 1,2,\hdots,d+1$, and $i=0,1,\hdots,\ell$. Then, from the proof of Proposition \ref{PropSparse}, we propose the following procedure to get $M(\la)$:

\fbox{\begin{minipage}{15cm}
		{\bf Procedure.} Construct $M(\la)$ with the following requirements:
		
		\vspace{0.2cm}
		
		{\bf Case 1:} $A = A_1$ (only if $\ell=1$):
		
		\vspace{0.2cm}
		
		\hspace{0.3cm} {\rm ($1$-i)} Place the coefficients of $P(\la)$ on the main diagonal of $M(\la)$ fulfilling the {\rm AS} condition.
		
		\vspace{0.2cm}
		
		\hspace{0.3cm} {\rm ($1$-ii)} (Optionally). Add other nonzero block entries with the constraint:
		\begin{equation*}
			[M_i]_{s,t} = -\cM_{A_1}[\pm [M_i]_{t,s}]^\star = \mp [[M_i]_{t,s}]^\star.
		\end{equation*}
		
		\vspace{0.1cm}
		
		{\bf Case 2:} $A = A_2$ (only if $\ell=1$):
		
		\vspace{0.2cm}
		
		\hspace{0.3cm} {\rm ($2$-i)} Place the coefficients of $P(\la)$ on the main diagonal of $M(\la)$ fulfilling the {\rm ASS} condition.
		
		\vspace{0.2cm}
		
		\hspace{0.3cm} {\rm ($2$-ii)} (Optionally). Add other nonzero block entries with the constraint:
		\begin{equation*}
			[M_i]_{s,t} = -\cM_{A_2}[\pm [M_i]_{t,s}]^\star = \mp [[M_i]_{t,s}(-\la)]^\star.
		\end{equation*}
		
		\vspace{0.2cm}
		
		{\bf Case 3:} $A = A_3$:
		
		\vspace{0.2cm}
		
		\hspace{0.3cm} {\rm ($3$-i)} Place the coefficients of $P(\la)$ in $d+1$ nonzero block entries of $M(\la)$, for $\ell=1$, or $2d+1$, 
		
		\vspace{0.1cm}
		
		\hspace{0.3cm} for $\ell>1$, in such a way that $M(\la)$ satisfies: 
		
		\vspace{0.2cm}
		
		\hspace{1cm}(3-i.a) the {\rm DS} condition, and 
		
		\vspace{0.2cm}
		
		\hspace{1cm}(3-i.b) if $[M]_{s,t}$ is nonzero, then $[M]_{t,s}$ is also nonzero.
		
		\vspace{0.2cm}
		
		\hspace{0.3cm} {\rm ($3$-ii)} (Optionally). Add other nonzero block entries, namely $H_1,\hdots,H_r$, in the same diagonal
		
		\hspace{0.3cm} such that $\sum_{j=0}^r H_j = 0$, and, for all $H_j$: 
		\begin{equation*}
			[M_i]_{s,t} = -\cM_{A_3}[\pm [M_i]_{t,s}]^\star = \mp [[M_{\ell-i}]_{t,s}]^\star.
		\end{equation*}
	\end{minipage}}
	
	\vspace{0.4cm}

The procedure above provides the matrix polynomial $M(\la)$ that is used to construct structure-preserving strong $\ell$-ifications ${\cal L}(\la)$ as in \eqref{GeneralLlambda} of $\cM_A$-structured matrix polynomials of grade $k$, and that gives the minimum number of nonzero block entries in ${\cal L}(\la)$ (namely, the one in Definition \ref{DefSparse}). Conditions (1-i), (2-i), and (3-i) in the previous procedure are the ones imposed on the matrix polynomial $M(\la)$ to guarantee that not only $M(\la)$ has the minimum number of nonzero block entries (in the sense of Definition \ref{DefSparse}) but also that $\widetilde{M}(\la):=\frac{1}{2}\left(M + \cM_A[\pm M]^\star\right)(\la)$ has the smallest number of nonzero block entries. Some other nonzero block entries can appear in $M(\la)$, but they must cancel out in the matrix polynomial $\widetilde{M}(\la)$. This is exactly the purpose of conditions (1-ii), (2-ii), and (3-ii). Despite these conditions, the number of nonzero block entries in $M(\la)$ and $\widetilde M(\la)$ are not necessarily the same, since, when computing the polynomial $\widetilde M(\la)$, some additional nonzero block entries may arise, so the final construction could not be sparse in the sense of Definition \ref{DefSparse}. This is exactly what happens for certain structures, as the following result shows.

\begin{Prop}\label{CorollarySparse}
	Let $P(\la)$ be as in the statement of Theorem {\rm\ref{StructLgeneral}}. Then, for $\ell>1$, any $\star$-(skew-) symmetric or $\star$-alternating block-Kronecker degree-$\ell$ matrix polynomials satisfying Theorem {\rm\ref{StructLgeneral}} has, at least, $7d+1$ nonzero block entries.
\end{Prop}

\begin{proof}
We first consider the $\star$-(skew-)symmetric structure. In particular, it is associated to the M\"obius transformation $\cM_{A_1}$ (see Definition \ref{Structures-Mobius}) and $M(\la)$ satisfies the {\rm AS} condition for $P(\la)$ (see Theorem \ref{NoStructLlambdaTeo}). Recall that, for $\ell>1$, the second line of the {\rm AS} condition in Definition \ref{structures} (the identity on the coefficients $P_{\ell r +c}$) implies that all the $2d+1$ anti-diagonals of $M(\la)$ have, at least, a nonzero block entry containing coefficients of $P(\la)$. Moreover, the coefficients can be grouped in just one nonzero block entry for each block anti-diagonal, as we have seen in the proof of Proposition \ref{PropSparse}. However, when we compute the matrix polynomial $\widetilde{M}(\la) = \frac{1}{2}\left(M + \cM_{A_1}[\pm M]^\star\right)(\la) = \frac{1}{2}\left(M \pm M^\star\right)(\la)$, all nonzero block entries outside the main diagonal of $M(\la)$ have a copy in $\widetilde{M}(\la)$, because of the term $M^\star(\la)$. Then, $\widetilde M(\la)$ will have, at least, $d$ more nonzero block entries than $M(\la)$, since the number of nonzero block entries of $M(\la)$ that can be placed on the main diagonal is, at most, $d+1$. In the best case, the $\star$-(skew-)symmetric block-Kronecker degree-$\ell$ matrix polynomial, has, at least, $4d+ 2d+1 +d = 7d+1$ nonzero block entries, where the $4d$ nonzero block entries come from the $(1,2)$ and $(2,1)$ big blocks in \eqref{GeneralLlambda}.  
	
The same reasoning as above can be applied for the $\star$-alternating structure, associated to the M\"obius transformation $\cM_{A_2}$ (see Definition \ref{Structures-Mobius}), where $M(\la)$ satisfies the {\rm AS} (resp. {\rm ASS}) condition for $P(\la)$ when $\ell$ is even (resp. odd) (see Theorem \ref{NoStructLlambdaTeo}).
\end{proof}

As a consequence of Proposition \ref{CorollarySparse}, there are no sparse (in the sense of Definition \ref{DefSparse}) structured block-Kronecker strong $\ell$-ifications, for $\ell>1$, and for the $\star$-(skew-)symmetric and $\star$-alternating structures.

We want to recall the recent work \cite{bueno2018explicit}, where the authors present four families of symmetric matrix pencils which, under some generic conditions, are block minimal bases pencils and strong linearizations of a given matrix polynomial. The first family, denoted by $\mathcal{O}_1^P(\la)$ and introduced in equation (4.17) in \cite{bueno2018explicit}, is the only one which is sparse in the sense of Definition \ref{DefSparse}, and the $(1,1)$ big block $\widetilde{M}(\la)$ of the strong linearization has all its nonzero block entries on the main diagonal. 

In the following example we illustrate the procedure presented after Definition \ref{DefSparse}, by explicitly constructing sparse structure-preserving strong linearizations for $\star$-skew-symmetric, $\star$-even or $\star$-anti-palindromic matrix polynomials of grade $7$, and sparse structure-preserving strong quadratifications for $\star$-palindromic matrix polynomials of grade $14$.

\begin{Exa}\label{ExampleNoDuplicSparse}
{\rm
	Let $P(\la) = \sum_{j=0}^{7} \la^j P_j \in \FF[\la]^{n \times n}$ be a matrix polynomial of grade $7$. Let us consider the following three matrix pencils:
	\begin{equation*}
		M_1(\la) := {\footnotesize \left[\begin{array}{cccc} P_6 + \la P_7 & 0 & 0 & 0 \\ 0 & P_4 + \la P_5 & 0 & 0 \\ 0 & 0 & P_2 + \la P_3 & 0 \\ 0 & 0 & 0 & P_0 + \la P_1 \end{array}\right]},
	\end{equation*}
	\begin{equation*}
	\begin{split}
		M_2(\la) &:= {\footnotesize \left[\begin{array}{cccc} -(P_6 + \la P_7) & 0 & 0 & 0 \\ 0 & P_4 + \la P_5 & 0 & -Q^\star(-\la) \\ 0 & 0 & -(P_2 + \la P_3) & 0 \\ 0 & Q(\la) & 0 & P_0 + \la P_1 \end{array}\right]},\ {\rm and} \\
		M_3(\la) &:= {\footnotesize\left[\begin{array}{cccc} 0 & P_2 + \la P_3 & 0 & P_0 + \la P_1 \\ P_4 + \la P_5 & 0 & 0 & 0 \\ 0 & 0 & 0 & 0 \\ P_6 + \la P_7 & 0 & 0 & 0 \end{array}\right]},
	\end{split}
	\end{equation*}
	where $Q(\la)$ is any matrix polynomial with coefficients in $\FF[P_0,\hdots,P_k]$ (in order for the pencil $M_2(\la)$ to provide a generalized companion linearization). 
	
	Looking at the procedure presented after Definition {\rm\ref{DefSparse}}, it is straightforward to check that $M_1(\la)$ satisfies {\rm ($1$-i)} and {\rm ($1$-ii)}, $M_2(\la)$ satisfies {\rm ($2$-i)} and {\rm ($2$-ii)}, and $M_3(\la)$ satisfies {\rm ($3$-i)} and {\rm ($3$-ii)}.
	
	Then, if $A_i$, for $i=1,2,3$, are the matrices in Definition {\rm\ref{Structures-Mobius}}, and $P(\la)$ is $\star$-skew-symmetric, $\star$-even, or $\star$-anti-palindromic, respectively (namely, $P(\la)$ is $\cM_{A_i}$-structured), the matrix pencils $\frac{1}{2}\left(M_i + \cM_{A_i}[\pm M_i]^\star\right)(\la)$ are given by 
	\begin{equation*}
		\begin{split}
			\frac{1}{2}\left(M_1 - M_1^\star\right)(\la) :=& {\footnotesize \left[\begin{array}{cccc} P_6 + \la P_7 & 0 & 0 & 0 \\ 0 & P_4 + \la P_5 & 0 & 0 \\ 0 & 0 & P_2 + \la P_3 & 0 \\ 0 & 0 & 0 & P_0 + \la P_1 \end{array}\right]},\\
			\frac{1}{2}\left(M_2(\la) + M_2^\star(-\la)\right) :=& {\footnotesize \left[\begin{array}{cccc} -(P_6 + \la P_7) & 0 & 0 & 0 \\ 0 & P_4 + \la P_5 & 0 & 0 \\ 0 & 0 & -(P_2 + \la P_3) & 0 \\ 0 & 0 & 0 & P_0 + \la P_1 \end{array}\right]},\ {\rm and}\\
			\frac{1}{2}\left(M_3 - {\rm rev}_1(M_3)^\star\right)(\la) :=& {\footnotesize \left[\begin{array}{cccc} 0 & P_2 + \la P_3 & 0 & P_0 + \la P_1 \\ P_4 + \la P_5 & 0 & 0 & 0 \\ 0 & 0 & 0 & 0 \\ P_6 + \la P_7 & 0 & 0 & 0 \end{array}\right]}.
		\end{split}
	\end{equation*}
	Then, by Theorem {\rm\ref{StructLgeneral}}, the $\cM_{A_i}$-structured block-Kronecker matrix pencils
	\begin{equation*}
		\mathcal{L}_{\cM_{A_i}}(\la) := \left[\begin{array}{c|c} \frac{1}{2}\left(M_i + \cM_{A_i}[\pm M_i]^\star\right)(\la) &  \cM_{A_i}[\pm L_3]^\top(\la) \otimes I_n \\ \hline L_3(\la) \otimes I_n & 0 \end{array}\right],\ {\rm for}\ i=1,2,3,
	\end{equation*}
	are sparse $\cM_{A_i}$-structured strong linearizations of $P(\la)$ whenever $P(\la)$ is $\cM_{A_i}$-structured.
	
	Now, let $P(\la) = \sum_{j=0}^{14} \la^j P_j \in \FF[\la]^{n \times n}$ be a $\star$-palindromic matrix polynomial of grade $14$. Let us consider the following two quadratic matrix polynomials:
	\begin{equation*}
		\begin{split}
			M_4(\la) & := {\footnotesize \left[\begin{array}{cccc} 0 & \la P_5 + \la^2 P_6 & Q(\la) & P_0 + \la P_1 + \la^2 P_2 \\ P_8 + \la P_9 & 0 & 0 & -Q(\la) + \la P_3 + \la^2 P_4 \\ -{\rm rev}_2Q^\star(\la) & 0 & \la P_7 & 0 \\ P_{12} + \la P_{13} + \la^2 P_{14} & {\rm rev}_2Q^\star(\la) + P_{10}+\la P_{11} & 0 & 0 \end{array}\right]},\ {\rm and} \\
			M_5(\la) &:= {\footnotesize \left[\begin{array}{cccc} 0 & \la P_5 + \la^2 P_6 & 0 & P_0 + \la P_1 + \la^2 P_2 \\ P_8 + \la P_9 & 0 & 0 & \la P_3 + \la^2 P_4 \\ 0 & 0 & \la P_7 & 0 \\ \la P_{13} + \la^2 P_{14} & P_{10} + \la P_{11} + \la^2 P_{12} & 0 & 0 \end{array}\right]}, 
		\end{split}
	\end{equation*}
	where $Q(\la)$ is any matrix polynomial with coefficients in $\FF[P_0,\hdots,P_k]$, in order for the construction to be a generalized companion $\ell$-ification.
	
	Notice that, looking at the procedure presented after Definition {\rm\ref{DefSparse}}, $M_4(\la)$ and $M_5(\la)$ satisfy {\rm ($3$-i)} and {\rm ($3$-ii)}. Now, computing the quadratic matrix polynomials $\frac{1}{2}\left(M_i + \cM_{A_3}[M_i]^\star\right)(\la)$, for $i=4,5$, we obtain
	\begin{equation*}
		\begin{split}
			\frac{1}{2}\left(M_4 + {\rm rev}_2 (M_4)^\star\right)(\la) & := {\footnotesize \left[\begin{array}{cccc} 0 & \la P_5 + \la^2 P_6 & 0 & P_0 + \la P_1 + \la^2 P_2 \\ P_8 + \la P_9 & 0 & 0 & \la P_3 + \la^2 P_4 \\ 0 & 0 & \la P_7 & 0 \\ P_{12} + \la P_{13} + \la^2 P_{14} & P_{10}+\la P_{11} & 0 & 0 \end{array}\right]},\ {\rm and\ } \\
			\frac{1}{2}\left(M_5 + {\rm rev}_2 (M_5)^\star\right)(\la) & := {\footnotesize \left[\begin{array}{cccc} 0 & \la P_5 + \la^2 P_6 & 0 & P_0 + \la P_1 + \la^2 P_2/2 \\ P_8 + \la P_9 & 0 & 0 & P_2/2 + \la P_3 + \la^2 P_4 \\ 0 & 0 & \la P_7 & 0 \\ P_{12}/2 + \la P_{13} + \la^2 P_{14} & P_{10} + \la P_{11} + \la^2 P_{12}/2 & 0 & 0 \end{array}\right]}.
		\end{split}
	\end{equation*} 
	By Theorem {\rm\ref{StructLgeneral}}, the $\star$-palindromic block-Kronecker quadratic matrix polynomials
	\begin{equation*}
		\mathcal{L}_{{\cal P},i}(\la) := \begin{bmatrix} \frac{1}{2}\left(M_i + {\rm rev}_2 (M_i)^\star\right)(\la) & {\rm rev}_2 (L_3)^\star(\la^2) \otimes I_n \\ L_3(\la^2) \otimes I_n & 0 \end{bmatrix},\ {\rm for}\ i=4,5,
	\end{equation*}
	are sparse $\star$-palindromic strong quadratifications of $P(\la)$ whenever $P(\la)$ is a $\star$-palindromic polynomial.
}
\end{Exa}

After looking at the constructions in Example \ref{ExampleNoDuplicSparse}, a couple of interesting observations raise up. The first one is that even though $\mathcal{L}_{{\cal P},5}(\la)$ is sparse (according to Definition \ref{DefSparse}), it contains duplicated coefficients of $P(\la)$ (in particular, $P_2$ and $P_{12}$ appear twice). The second one is that all the nonzero block entries of the $(1,1)$ big block in $\mathcal{L}_{\cM_{A_i}}(\la)$, for $i=1,2$, are on the main diagonal. However, this is not true for $\mathcal{L}_{\cM_{A_3}}(\la)$. %Let us recall that they are all sparse $\ell$-ifications, according to Definition \ref{DefSparse}.

The block-Kronecker minimal bases $\ell$-ifications in Theorem \ref{StructLgeneral} are particular cases of the strong block minimal bases $\ell$-ifications. These $\ell$-ifications allow for more flexibility in the $(1,2)$ and $(2,1)$ big blocks (the ones containing the minimal bases). As an example, we are going to construct a block minimal bases structure-preserving strong $\ell$-ification as in \eqref{StructLSBlock} that contains some additional invertible matrices in its $(1,2)$ and $(2,1)$ big blocks. The idea is to replace the block $L_d(\la^\ell)\otimes I_n$ in \eqref{StructLSBlockKronecker} by another minimal basis whose dual minimal basis is still the block-Kronecker minimal basis $\Lambda_d(\la^\ell)\otimes I_n$ in \eqref{Lambdad}. Then, the construction provides an $\cM_A$-structured strong block minimal bases grade-$\ell$ matrix polynomial as in \eqref{StructLSBlock}, which is a strong $\ell$-ification of a given $\cM_A$-structured matrix polynomial as in \eqref{StructPlambdaKronecker}. This is shown in Example \ref{ExampleInvMatrices}, where, in particular, we construct a structure-preserving strong cubification $(\ell=3)$ of a given $\star$-symmetric matrix polynomial of grade $21$.

\begin{Exa}\label{ExampleInvMatrices}
{\rm
	Let $P(\la) = \sum_{j=0}^{21} \la^j P_j \in \FF[\la]^{n \times n}$ be a $\star$-symmetric matrix polynomial of grade $21$, and let $X,Y \in \FF^{n \times n}$ be invertible matrices. Let us consider the following matrix pencils:
	\begin{equation*}
	\begin{split}
	M(\la) :=& {\footnotesize \left[\begin{array}{cccc} P_{18} + \la P_{19} + \la^2 P_{20} + \la^3 P_{21} & 0 & 0 & P_9 + \la P_{10} + \la^2 P_{11} + \la^3 P_{12} \\ P_{15}+ \la P_{16} + \la^2 P_{17} & \la P_{13} + \la^2 P_{14} & 0 & 0 \\ 0 & 0 & P_6 + \la P_7 + \la^2 P_8 & P_3 + \la P_4 + \la^2 P_5 \\ 0 & 0 & 0 & P_0 + \la P_1 + \la^2 P_2 \end{array}\right]},\ {\rm and}\\
	K(\la) :=& {\footnotesize \left[\begin{array}{cccc} -X & \la X & 0 & 0 \\ 0 & -Y & \la Y & 0 \\ 0 & 0 & -I_n & \la I_n \end{array}\right]}.
	\end{split}
	\end{equation*}
	$M(\la)$ satisfies the {\rm AS} condition for $P(\la)$ (Definition {\rm\ref{structures}}). If $K(\la^3)$ is the matrix polynomial $K(\la)$ evaluated at $\la^3$, notice that both $K(\la)$ and $K(\la^3)$ are minimal bases. Moreover, the matrix polynomial $\Lambda_3^\top(\la^3) \otimes I_n := \begin{bmatrix} \la^9 I_n & \la^6 I_n & \la^3 I_n & I_n \end{bmatrix}$ is a minimal basis dual to $K(\la^3)$ (Definition {\rm\ref{dualminbases}}).
	
	By Theorem {\rm\ref{Ma-strucstrong}}, the $\star$-symmetric strong block minimal bases cubic matrix polynomial
	\begin{equation*}
	\mathcal{L}_{{\cal S}}(\la) := \begin{bmatrix} \frac{1}{2}\left(M + M^\star\right)(\la) & K^\star(\la^3) \\ K(\la^3) & 0 \end{bmatrix}
	\end{equation*}
	is a $\star$-symmetric strong cubification of $P(\la):= \left(\Lambda_3^\top(\la^3)\otimes I_n\right)\left(\frac{1}{2}(M + M^\star)(\la)\right)\left(\Lambda_3(\la^3)\otimes I_n\right)$.
}
\end{Exa}

\section{There are no $\boldsymbol{\cM_A}$-structured companion quadratifications for $\boldsymbol{\cM_A}$-structured quartic matrix polynomials}\label{SectionNoStructCompLific}

The wide family of structure-preserving strong $\ell$-ifications of a given $\cM_A$-structured matrix polynomial $P(\la)$ of grade $k$, constructed in Theorem \ref{StructLgeneral}, deals with matrix polynomials which are generalized companion $\ell$-ifications of $P(\la)$ (Definition \ref{GenCompLification}). However, in this section we focus on companion $\ell$-ifications (Definition \ref{CompLification}). In particular, we will prove, in Theorem \ref{NoStructCompLific}, that, for all the structures in Definition \ref{Structures-Mobius}, there are no $\cM_A$-structured companion quadratifications for any $\cM_A$-structured quartic matrix polynomial $P(\la)$.

The notion of companion $\ell$-ification is more restrictive than the one of generalized companion $\ell$-ification (as the names suggest). However, in the literature, the most well-known families of linearizations (that is, $\ell$-ifications with $\ell=1$) are companion. We refer, in particular, to the classical Frobenius companion pencils and to all families of Fiedler-like pencils (Fiedler, generalized Fiedler, and Fiedler pencils with repetition). Even the more recent classes of block-Kronecker linearizations \cite{dopico2018block} and $\ell$-ifications \cite{dopico2019block} contain many companion linearizations/$\ell$-ifications. Moreover, it is natural to focus on companion $\ell$-ifications, since they are the simplest structures which are strong $\ell$-ifications and valid for all polynomials. 

In \cite[Th. 7.20]{de2014spectral}, it was proved that there are no structured generalized companion pencils for structured matrix polynomials of even grade $k$, for any of the following structures: $\star$-symmetric, $\star$-alternating, and $\star$-palindromic. We have seen that, however, there are structured generalized companion $\ell$-ifications when $k = (2d+1)\ell$. Moreover, our constructions include structured companion $\ell$-ifications for any of the structures considered in this paper. 

It is natural to ask whether for $k = (2d) \ell$ there are structured (generalized) companion $\ell$-ifications or not. 

In \cite{huang2011palindromic}, a family of companion-like palindromic quadratifications for palindromic matrix polynomials of any even degree has been presented. This family, however, is not generalized companion, since it involves the transpose (or conjugate transpose) of some of the coefficients of the polynomial. Nonetheless, based in this construction, it is not difficult to construct a generalized companion quadratification for quartic matrix polynomials which is $\star$-palindromic whenever the polynomial is. More precisely:
\begin{equation*}
L(\la) = {\footnotesize \left[\begin{array}{cc} P_1 + \la (P_2 - I - P_0P_4) + \la^2 P_3 & I + \la^2 P_4 \\ P_0 + \la^2 I & -\la I \end{array}\right]}
\end{equation*}
is a strong quadratification of $P(\la) = \sum_{j=0}^4 \la^j P_j$. In order to see that it is a quadratification just perform the following elementary block-column and block-row operations:
\begin{equation*}
\begin{split}
L(\la) & \xrightarrow{C_{12}(\la I)} {\footnotesize \left[\begin{array}{cc} P_1 + \la (P_2 - P_0P_4) + \la^2 P_3 + \la^3 P_4 & I + \la^2 P_4 \\ P_0 & -\la I \end{array}\right]} \xrightarrow{R_{12}(\la P_4)} {\footnotesize \left[\begin{array}{cc} P_1 + \la P_2 + \la^2 P_3 + \la^3 P_4 & I \\ P_0 & -\la I \end{array}\right]} \\
& \xrightarrow{R_{21}(\la I)} {\footnotesize \left[\begin{array}{cc} P_1 + \la P_2 + \la^2 P_3 + \la^3 P_4 & I \\ P(\la) & 0 \end{array}\right]} \xrightarrow{C_{12}(-(P_1 + \la P_2 + \la^2 P_3 + \la^3 P_4))} {\footnotesize \left[\begin{array}{cc} 0 & I \\ P(\la) & 0 \end{array}\right]} \xrightarrow{R_{12}} {\footnotesize \left[\begin{array}{cc} P(\la) & 0 \\ 0 & I \end{array}\right]}.
\end{split}
\end{equation*}
To see that it is strong, we perform the following elementary block-column and block-row operations on the polynomial ${\rm rev}_2L(\la)$:
\begin{equation*}
\begin{split}
{\rm rev}_2L(\la) &= {\footnotesize \left[\begin{array}{cc} P_3 + \la (P_2 - I - P_0P_4) + \la^2 P_1 & P_4 + \la^2 I \\ I + \la^2 P_0 & -\la I \end{array}\right]} \xrightarrow{R_{12}(\la I)} {\footnotesize \left[\begin{array}{cc} P_3 + \la (P_2 - P_0P_4) + \la^2 P_1 + \la^3 P_0 & P_4 \\ I + \la^2 P_0 & -\la I \end{array}\right]}  \\
& \xrightarrow{C_{12}(\la P_0)} {\footnotesize \left[\begin{array}{cc} P_3 + \la P_2 + \la^2 P_1 + \la^3 P_0 & P_4 \\ I & -\la I \end{array}\right]} \xrightarrow{C_{21}(\la I)} {\footnotesize \left[\begin{array}{cc} P_3 + \la P_2 + \la^2 P_1 + \la^3 P_0 & {\rm rev}_4P(\la) \\ I & 0 \end{array}\right]} \\ 
& \xrightarrow{R_{12}(-(P_3 + \la P_2 + \la^2 P_1 + \la^3 P_0))} {\footnotesize \left[\begin{array}{cc} 0 & {\rm rev}_4P(\la) \\ I & 0 \end{array}\right]} \xrightarrow{C_{12}} {\footnotesize \left[\begin{array}{cc} {\rm rev}_4P(\la) & 0 \\ 0 & I \end{array}\right]}.
\end{split}
\end{equation*}
Finally, it is straightforward to see that $L(\la)$ is $\star$-palindromic whenever $P(\la)$ is $\star$-palindromic, since, in this case $P_0 = P_4^\star$, $P_1 = P_3^\star$, and $P_2 = P_2^\star$, so:
\begin{equation*}
\begin{split}
L^\star(\la) =&{\footnotesize \left[\begin{array}{cc} P_1^\star + \la (P_2^\star - I - P_4^\star P_0^\star) + \la^2 P_3^\star & P_0^\star + \la^2 I \\ I + \la^2 P_4^\star & -\la I \end{array}\right]} = \\
=&{\footnotesize \left[\begin{array}{cc} P_3 + \la (P_2 - I - P_0 P_4) + \la^2 P_1 & P_4 + \la^2 I \\ I + \la^2 P_0 & -\la I \end{array}\right]} = {\rm rev}_2 L(\la).
\end{split}
\end{equation*}
As a consequence, there are $\star$-palindromic generalized companion quadratifications of $\star$-palindromic quartic matrix polynomials. However, if we remove the adjective ``generalized" and we restrict ourselves to companion quadratifications, Theorem \ref{NoStructCompLific} tells us that this is no longer the case. %In particular, this is not the case for any companion $\ell$-ification for matrix polynomials of degree $k=2\ell$ (with $\ell$ even). 

\begin{Teo}\label{NoStructCompLific}
	Let $P(\la) = \sum_{j=0}^4 \la^j P_j \in \FF[\la]^{n \times n}$ be an $\cM_A$-structured quartic matrix polynomial, with $A$ being any of the matrices in Definition {\rm\ref{Structures-Mobius}}. Then, there are no $\cM_A$-structured companion quadratifications for $P(\la)$.
\end{Teo} 

\begin{proof}
	We focus on the scalar case (namely, $n=1$) for simplicity, because in the case of existing a companion quadratification like the one mentioned in the statement, it should be valid for all values of $n$. We procced by contradiction. Notice that, for scalar polynomials, $p^\star(\la)$ can be either $p(\la)$ (if $\star = \top$) or $\overline{p}(\la)$ (if $\star = \ast$). In addition, when $p(\la)$ is scalar, condition (b) in Definition \ref{CompLification} can be replaced by the following one:
	\begin{equation}\label{detL=ap}
	\det L(\la) = \alpha p(\la) = \alpha \sum_{j=0}^k \la^j p_j,
	\end{equation}
	for some $0 \neq \alpha \in \FF$.
	
	We are going to prove that there are no $\cM_A$-structured companion quadratifications for $\cM_A$-structured quartic scalar polynomials. So let $p(\la) = \sum_{j=0}^4 \la^j p_j$ be a quartic scalar polynomial, with $p_j \in \FF$, for $j=0,1,\hdots,4$, and let
	\begin{equation*}
	L(\la) = \begin{bmatrix} l_{1,1}(\la) & l_{1,2}(\la) \\ l_{2,1}(\la) & l_{2,2}(\la) \end{bmatrix}
	\end{equation*}
	be a quadratic matrix polynomial, where $l_{s,t}(\la) = \sum_{i=0}^2 \la^i (l_i)_{s,t}$, for $s,t = 1,2$, (with $(l_i)_{s,t} \in \FF$ for all values of $s$, $t$, and $i$). 
	
	We consider only the $\star$-symmetric case (since the procedure for the $\star$-skew-symmetric and $\star$-alternating cases are very similar to this one) and the $\star$-palindromic case (since the procedure for the $\star$-anti-palindromic case is very similar to this one).
	
	First, we consider the $\star$-symmetric case. More precisely, we assume that $L(\la)$ is $\star$-symmetric whenever $p(\la)$ is $\star$-symmetric, and that $L(\la)$ is a companion quadratification of $p(\la)$. If $p(\la)$ is $\star$-symmetric, it satisfies $p^\star(\la) = p(\la)$, namely,
	\begin{equation}\label{CondTSymmP}
	p_j^\star = p_j,\ {\rm for}\ 0 \leq j \leq 4,
	\end{equation}
	whereas, if $L(\la)$ is $\star$-symmetric, it satisfies $L^\star(\la) = L(\la)$, that is,
	\begin{equation}\label{CondTSymmL}
	\begin{split}
	(l_i)_{s,s}^\star & = (l_i)_{s,s},\ {\rm for}\ i=0,1,2,\ {\rm and}\ s=1,2,\\
	(l_i)_{1,2}^\star & = (l_i)_{2,1},\ {\rm for}\ i=0,1,2.
	\end{split}
	\end{equation}
	Now, equation \eqref{detL=ap} is equivalent to
	\begin{equation}\label{detLS=alphap}
	{\rm det} \begin{bmatrix} (l_0)_{1,1} + \la (l_1)_{1,1} + \la^2 (l_2)_{1,1} & (l_0)_{1,2} + \la (l_1)_{1,2} + \la^2 (l_2)_{1,2} \\ (l_0)_{1,2}^\star + \la (l_1)_{1,2}^\star + \la^2 (l_2)_{1,2}^\star & (l_0)_{2,2} + \la (l_1)_{2,2} + \la^2 (l_2)_{2,2} \end{bmatrix} = \alpha \sum_{j=0}^4 \la^j p_j.
	\end{equation}
	Spanning the determinant in the left-hand side of \eqref{detLS=alphap}, and equating the coefficients of the terms in $\la$ with the same degree in both sides, we get, for the coefficients of $\la^0$ and $\la^4$, the following equations:
	\begin{align}
	\alpha p_0 & = (l_0)_{1,1}(l_0)_{2,2} - (l_0)_{1,2}(l_0)_{1,2}^\star, \label{p0S}\\
	%\alpha p_2 & = l_0^{(1,1)}l_2^{(2,2)} + l_1^{(1,1)}l_1^{(2,2)} + l_2^{(1,1)}l_0^{(2,2)} - l_0^{(1,2)}l_2^{(1,2)\star} - l_1^{(1,2)}l_1^{(1,2)\star} - l_2^{(1,2)}l_0^{(1,2)\star}, \label{a2S} \\
	\alpha p_4 & = (l_2)_{1,1}(l_2)_{2,2} - (l_2)_{1,2}(l_2)_{1,2}^\star. \label{p4S}
	\end{align}
	Looking at \eqref{p0S}, the coefficient $p_0$ must appear in the right-hand side of the equation with degree $1$. In order for this to happen, and taking into account that $L(\la)$ is a companion quadratification (namely, $(l_i)_{s,t}$ is either $\pm \beta$ or $\pm \beta p_j$, for some $\beta \in \FF$ and $j=0,\hdots,4$), the only possibility is that it appears in the product $(l_0)_{1,1}(l_0)_{2,2}$. As a consequence, it must be $(l_0)_{1,2}=0$. Then, there are only two possibilities for \eqref{p0S}, namely
	\begin{enumerate}
		\item[($S_{0.1}$):] $(l_0)_{1,1} = \beta_0$ and $(l_0)_{2,2} = \alpha_0 p_0$, where $\beta_0\alpha_0 = \alpha \neq 0$, and $(l_0)_{1,2}=0$.
		\item[($S_{0.2}$):] $(l_0)_{1,1} = \alpha_0p_0$ and $(l_0)_{2,2} = \beta_0$, where $\beta_0\alpha_0 = \alpha \neq 0$, and $(l_0)_{1,2}=0$.
	\end{enumerate}
	Using the same argument for the coefficient $p_4$ in \eqref{p4S}, there are another two possibilities, namely,
	\begin{enumerate}
		\item[($S_{4.1}$):] $(l_2)_{1,1} = \beta_4$ and $(l_2)_{2,2} = \alpha_4 p_4$, where $\beta_4\alpha_4 = \alpha \neq 0$, and $(l_2)_{1,2}=0$.
		\item[($S_{4.2}$):] $(l_2)_{1,1} = \alpha_4 p_4$ and $(l_2)_{2,2} = \beta_4$, where $\beta_4\alpha_4 = \alpha \neq 0$, and $(l_2)_{1,2}=0$.
	\end{enumerate}
	Now, spanning the determinant in the left-hand side of \eqref{detLS=alphap} and equating the coefficients of $\la$ and $\la^3$, we get
	\begin{align}
	\alpha p_1 & = (l_0)_{1,1}(l_1)_{2,2} + (l_1)_{1,1}(l_0)_{2,2} - (l_0)_{1,2}(l_1)_{1,2}^\star - (l_1)_{1,2}(l_0)_{1,2}^\star, \label{p1S}\\
	\alpha p_3 & = (l_1)_{1,1}(l_2)_{2,2} + (l_2)_{1,1}(l_1)_{2,2} - (l_1)_{1,2}(l_2)_{1,2}^\star - (l_2)_{1,2}(l_1)_{1,2}^\star. \label{p3S}
	\end{align}
	Depending on ($S_{0.1}$) and ($S_{0.2}$) for $p_0$, or ($S_{4.1}$) and ($S_{4.1}$) for $p_4$, \eqref{p1S} and \eqref{p3S} can be rewritten as follows:
	\begin{align}
	& \left\{\begin{array}{c}{\rm for}\ (S_{0.1}): \alpha p_1 = \beta_0(l_1)_{2,2} + (l_1)_{1,1}\alpha_0p_0, \\
	{\rm for}\ (S_{0.2}): \alpha p_1 = \alpha_0p_0(l_1)_{2,2} + (l_1)_{1,1}\beta_0. \end{array}\right. \label{p1S12} \\
	& \left\{\begin{array}{c} {\rm for}\ (S_{4.1}): \alpha p_3 = (l_1)_{1,1}\alpha_4p_4 + \beta_4(l_1)_{2,2}, \\
	{\rm for}\ (S_{4.2}): \alpha p_3 = (l_1)_{1,1}\beta_4 + \alpha_4p_4(l_1)_{2,2}. \end{array}\right. \label{p3S12}
	\end{align}
	Looking at \eqref{p1S12}, the coefficient $p_1$ must appear in the right-hand side of both equations with degree $1$, and the only possibility is that it appears in either the term $\beta_0(l_1)_{2,2}$ (for ($S_{0.1}$)), or $(l_1)_{1,1}\beta_0$ (for ($S_{0.2}$)). In the first case, it must be $(l_1)_{1,1} = 0$, whereas in the second case it must be $(l_1)_{2,2} = 0$. As a consequence, we get the following two possibilities:
	\begin{enumerate}
		\item[($S_{1.1}$):] {\rm ($S_{0.1}$)}, together with $(l_1)_{2,2} = \alpha_1 p_1$, where $\alpha_1\beta_0 = \alpha \neq 0$, and $(l_1)_{1,1}=0$.
		\item[($S_{1.2}$):] {\rm ($S_{0.2}$)}, together with $(l_1)_{1,1} = \alpha_1 p_1$, where $\alpha_1\beta_0 = \alpha \neq 0$, and $(l_1)_{2,2}=0$.
	\end{enumerate}
	Replacing the terms obtained in the two precedent cases in any of the two equations in \eqref{p3S12}, we get that the coefficient $p_3$ cannot appear in the right-hand side of any of the equations with degree $1$, in contradiction with the left-hand side.
%	Using the same argument for the coefficient $p_3$ {in \eqref{p3S12}}, we get the following two possibilities:
%	\begin{enumerate}
%		\item[($S_{3.1}$):] {\rm ($S_{4.1}$)}, together with $(l_1)_{2,2} = \alpha_3 p_3$, where $\alpha_3\beta_4 = \alpha \neq 0$, and $(l_1)_{1,1}=0$.
%		\item[($S_{3.2}$):] {\rm ($S_{4.2}$)}, together with $(l_1)_{1,1} = \alpha_3p_3$, where $\alpha_3\beta_4 = \alpha \neq 0$, and $(l_1)_{2,2}=0$.
%	\end{enumerate}
%	Replacing the terms obtained in the precedent cases, and equating the coefficients of $\la^2$ in \eqref{detLS=alphap}, we get:
%	\begin{align}
%	{\rm for}\ (S_{1.1})+(S_{3.1}): \alpha p_2 = \beta_0\alpha_4p_4 + \beta_4\alpha_0p_0 - (l_1)_{1,2}(l_1)_{1,2}^\star, \nonumber \\
%	{\rm for}\ (S_{1.1})+(S_{3.2}): \alpha p_2 = \beta_0\beta_4 + \alpha_4p_4\alpha_0p_0 - (l_1)_{1,2}(l_1)_{1,2}^\star, \nonumber \\
%	{\rm for}\ (S_{1.2})+(S_{3.1}): \alpha p_2 = \alpha_0p_0\alpha_4p_4 + \beta_4\beta_0 - (l_1)_{1,2}(l_1)_{1,2}^\star, \nonumber \\
%	{\rm for}\ (S_{1.2})+(S_{3.2}): \alpha p_2 = \alpha_0p_0\beta_4 + \alpha_4p_4\beta_0 - (l_1)_{1,2}(l_1)_{1,2}^\star. \nonumber
%	\end{align}
%	Looking at the above equations, the coefficient $p_2$ {\color{blue}cannot} appear in the right-hand side of any of them with degree $1$, in contradiction with the left-hand side.
	
	Therefore, none of the cases may occur, so there is no $\star$-symmetric companion quadratification for $\star$-symmetric quartic polynomials. 
	
	Now, we consider the $\star$-palindromic case. More precisely, we assume that $L(\la)$ is $\star$-palindromic whenever $p(\la)$ is $\star$-palindromic, and that $L(\la)$ is a companion quadratification of $p(\la)$. If $p(\la)$ is $\star$-palindromic, it satisfies $p^\star(\la) = {\rm rev}_4p(\la)$, namely, 
	\begin{equation}\label{condPalP}
	\begin{split}
	p_j^\star = p_{4-j},\ {\rm for}\ 0 \le j \le 2,
	\end{split}
	\end{equation}
	whereas, if $L(\la)$ is $\star$-palindromic, it satisfies
	\begin{equation*}
	L^\star(\la) := \begin{bmatrix} l_{1,1}^\star(\la) & l_{2,1}^\star(\la) \\ l_{1,2}^\star(\la) & l_{2,2}^\star(\la) \end{bmatrix} = {\rm rev}_2 L(\la) := \begin{bmatrix} {\rm rev}_2 l_{1,1}(\la) & {\rm rev}_2 l_{1,2}(\la) \\ {\rm rev}_2 l_{2,1}(\la) & {\rm rev}_2 l_{2,2}(\la) \end{bmatrix},\ {\rm that\ is}
	\end{equation*}
	\begin{equation}\label{condPalL}
	\begin{split}
	(l_i)_{s,s}^\star & = (l_{2-i})_{s,s},\ {\rm for}\ i = 0,1,\ {\rm and}\ s=1,2,\\
	(l_i)_{1,2}^\star & = (l_{2-i})_{2,1},\ {\rm for}\ i = 0,1.
	\end{split}
	\end{equation}
	Now, \eqref{detL=ap} is equivalent to
	\begin{equation}\label{detLP=alphap}
	{\rm det} \begin{bmatrix} (l_0)_{1,1} + \la (l_1)_{1,1} + \la^2 (l_0)_{1,1}^\star & (l_0)_{1,2} + \la (l_1)_{1,2} + \la^2 (l_2)_{1,2} \\ (l_2)_{1,2}^\star + \la (l_1)_{1,2}^\star + \la^2 (l_0)_{1,2}^\star & (l_0)_{2,2} + \la (l_1)_{2,2} + \la^2 (l_0)_{2,2}^\star \end{bmatrix} = \alpha \sum_{j=0}^4 \la^j p_j.
	\end{equation}
	Spanning the determinant in the left-hand side of \eqref{detLP=alphap}, and equating the coefficients for $\la^j$ and $\la^{4-j}$ in both sides, for $j=0,1,2$, we three identities in the form: $\alpha p_j = f_j$, and $\alpha p_{4-j} = f_j^\star$, for some function $f$ depending on the coefficients $(l_i)_{s,t}$, for $1 \le s,t \le 2$ and $0 \le i \le 2$. This, in particular, implies that, if $\star = \ast$, then it must be $\alpha = \overline{\alpha}$ (that is, $\alpha \in \mathbb{R}$). From the practical point of view, these two identities mean that the conditions we get for $p_j$ and $p_{4-j} = p_j^\star$ are exactly the same. Therefore, we will focus only on the equations for the coefficients $p_j$, for $j=0,1,2$. First, spanning the determinant in the left-hand side of \eqref{detLP=alphap} and equating the coefficients of $\la^0$, we get the following equation:
	\begin{align}
	\alpha p_0 & = (l_0)_{1,1}(l_0)_{2,2} - (l_0)_{1,2}(l_2)_{1,2}^\star. \label{p0P}
	\end{align}
	Looking at \eqref{p0P}, the coefficient $p_0$ must appear in the right-hand side of the equation with degree $1$, and there are three different possibilites:
	\begin{itemize}
		\item $(l_0)_{1,1}(l_0)_{2,2} = \alpha_1 p_0$ and $(l_0)_{1,2}(l_2)_{1,2}^\star = \gamma_1 p_0$, where $0 \neq \alpha_1 \in \FF$, $0 \neq \gamma_1 \in \FF$ and $\alpha_1 -\gamma_1 = \alpha$.
		\item $(l_0)_{1,1}(l_0)_{2,2} = \alpha p_0$ and $(l_0)_{1,2}(l_2)_{1,2}^\star = 0$.
		\item $(l_0)_{1,1}(l_0)_{2,2} = 0$ and $(l_0)_{1,2}(l_2)_{1,2}^\star = \alpha p_0$.	
	\end{itemize}
	These three possibilities can be subdivided in four cases each, depending on the possible values of the coefficients of $L(\la)$, namely:
	\begin{enumerate}
		\item[{\rm ($P_{0.1}$)}:] $(l_0)_{1,1} = \alpha_0 p_0$ and $(l_0)_{2,2} = \beta_0$, where $\beta_0\alpha_0 = \alpha_1 \neq 0$, $(l_0)_{1,2} = \gamma_0 p_0$ and $(l_2)_{1,2}^\star = \delta_0$, where $\delta_0\gamma_0 = \gamma_1 \neq 0$. 
		\item[{\rm ($P_{0.2}$)}:] $(l_0)_{1,1}= \alpha_0 p_0$ and $(l_0)_{2,2} = \beta_0$, where $\beta_0\alpha_0 = \alpha_1 \neq 0$, $(l_0)_{1,2} = \delta_0$ and $(l_2)_{1,2}^\star = \gamma_0 p_0$, where $\delta_0\gamma_0 = \gamma_1 \neq 0$. 
		\item[{\rm ($P_{0.3}$)}:] $(l_0)_{1,1} = \beta_0$ and $(l_0)_{2,2} = \alpha_0 p_0$, where $\beta_0\alpha_0 = \alpha_1 \neq 0$, $(l_0)_{1,2} = \gamma_0 p_0$ and $(l_2)_{1,2}^\star = \delta_0$, where $\delta_0\gamma_0 = \gamma_1 \neq 0$. 
		\item[{\rm ($P_{0.4}$)}:] $(l_0)_{1,1} = \beta_0$ and $(l_0)_{2,2} = \alpha_0 p_0$, where $\beta_0\alpha_0 = \alpha_1 \neq 0$, $(l_0)_{1,2} = \delta_0$ and $(l_2)_{1,2}^\star = \gamma_0 p_0$, where $\delta_0\gamma_0 = \gamma_1 \neq 0$. 
		\item[{\rm ($P_{0.5}$)}:] $(l_0)_{1,1} = \alpha_0 p_0$ and $(l_0)_{2,2} = \beta_0$, where $\beta_0\alpha_0 = \alpha \neq 0$, and $(l_0)_{1,2} = 0$.
		\item[{\rm ($P_{0.6}$)}:] $(l_0)_{1,1} = \alpha_0 p_0$ and $(l_0)_{2,2} = \beta_0$, where $\beta_0\alpha_0 = \alpha \neq 0$, and $(l_2)_{1,2}^\star = 0$.
		\item[{\rm ($P_{0.7}$)}:] $(l_0)_{1,1} = \beta_0$ and $(l_0)_{2,2} = \alpha_0 p_0$, where $\beta_0\alpha_0 = \alpha \neq 0$, and $(l_0)_{1,2} = 0$.
		\item[{\rm ($P_{0.8}$)}:] $(l_0)_{1,1} = \beta_0$ and $(l_0)_{2,2} = \alpha_0 p_0$, where $\beta_0\alpha_0 = \alpha \neq 0$, and $(l_2)_{1,2}^\star = 0$.
		\item[{\rm ($P_{0.9}$)}:] $(l_0)_{1,2} = \alpha_0 p_0$ and $(l_2)_{1,2}^\star = \beta_0$, where $\beta_0\alpha_0 = \alpha \neq 0$, and $(l_0)_{1,1} = 0$.
		\item[{\rm ($P_{0.10}$)}:] $(l_0)_{1,2} = \alpha_0 p_0$ and $(l_2)_{1,2}^\star = \beta_0$, where $\beta_0\alpha_0 = \alpha \neq 0$, and $(l_0)_{2,2} = 0$.
		\item[{\rm ($P_{0.11}$)}:] $(l_0)_{1,2} = \beta_0$ and $(l_2)_{1,2}^\star = \alpha_0 p_0$, where $\beta_0\alpha_0 = \alpha \neq 0$, and $(l_0)_{1,1} = 0$.
		\item[{\rm ($P_{0.12}$)}:] $(l_0)_{1,2} = \beta_0$ and $(l_2)_{1,2}^\star = \alpha_0 p_0$, where $\beta_0\alpha_0 = \alpha \neq 0$, and $(l_0)_{2,2} = 0$.
	\end{enumerate}
	We are going to consider each case separately. Let us directly look at the coefficient of $\la^2$ in \eqref{detLP=alphap}. We equate in both sides and use the particular conditions on the coefficients of $L(\la)$ obtained in ($P_{0.1}$)--($P_{0.12}$). We will focus only on ($P_{0.1}$), ($P_{0.5}$), and ($P_{0.9}$). The cases ($P_{0.2}$)--($P_{0.4}$) are analogous to ($P_{0.5}$), the cases ($P_{0.6}$)--($P_{0.8}$) are similar to ($P_{0.5}$), and ($P_{0.10}$)--($P_{0.12}$) are similar to ($P_{0.9}$).  
	\begin{align}
	{\rm for}\ (P_{0.1}): \alpha p_2 =& \alpha_0 p_0(\beta_0)^\star + (l_1)_{1,1}(l_1)_{2,2} + (\alpha_0 p_0)^\star\beta_0 - (\gamma_0 p_0)(\gamma_0 p_0)^\star - (l_1)_{1,2}(l_1)_{1,2}^\star - (\delta_0)^\star\delta_0, \label{p2P1} \\
	{\rm for}\ (P_{0.5}): \alpha p_2 =& \alpha_0 p_0(\beta_0)^\star + (l_1)_{1,1}(l_1)_{2,2} + (\alpha_0 p_0)^\star\beta_0 - (l_1)_{1,2}(l_1)_{1,2}^\star - (l_2)_{1,2}(l_2)_{1,2}^\star, \label{p2P5} \\
	{\rm for}\ (P_{0.9}): \alpha p_2 =& (l_1)_{1,1}(l_1)_{2,2} - (\alpha_0 p_0)(\alpha_0 p_0)^\star - (l_1)_{1,2}(l_1)_{1,2}^\star - (\beta_0)^\star\beta_0. \label{p2P9}
	\end{align}
	Looking at \eqref{p2P1} and \eqref{p2P5}, the coefficient $p_2$ must appear in the right-hand side of the equation with degree $1$ and the only possibility is that it appears in the product $(l_1)_{1,1}(l_1)_{2,2}$. As a consequence, the two terms with degree $1$ in $p_0$ can not be cancelled with any other term (notice that the sum of both terms is equal to $2\beta_0\alpha_0p_0$, if $\star = \top$, and $2{\rm Re}(\beta_0\alpha_0p_0)$, if $\star = \ast$, with $\beta_0\alpha_0  = \alpha_1 \neq 0$). Therefore, the cases ($P_{0.1}$) and ($P_{0.5}$) are not possible, and because of the symmetries in the equations when spanning \eqref{detLP=alphap}, the cases ($P_{0.2}$)--($P_{0.4}$) and ($P_{0.6}$)--($P_{0.8}$) are not possible since they give us analogous equations to \eqref{p2P1} or \eqref{p2P5}, respectively.
	
	Looking at \eqref{p2P9}, the coefficient $p_2$ must appear in the right-hand side of the equation with degree $1$ and the only possibility is that it appears in the product $(l_1)_{1,1}(l_1)_{2,2}$. In addition, not only the term containing the coefficient $p_0$, but also the constant term must cancel out, and both terms cannot be cancelled out with the remaining term $(l_1)_{1,2}(l_1)_{1,2}^\star$. As a consequence, ($P_{0.9}$) is not possible, and then ($P_{0.10}$)--($P_{0.12}$) are also not possible because symmetries.
	
	Therefore, none of the cases may occur, so we get a contradiction. Then, there is no $\star$-palindromic companion quadratification for $\star$-palindromic quartic polynomials.
\end{proof}

\begin{Rem}
{\rm
	Note that, in the proof of Theorem \ref{NoStructCompLific}, we have not used the fact that $L(\la)$ is a strong quadratification. Only the property that $L(\la)$ is a quadratification, for all values of $p_0,\hdots,p_4$, is needed.
}
\end{Rem}

In \cite[Th. 2.7]{mackey2006structured}, the authors proved that $\star$-alternating and $\star$-(anti-)palindromic matrix polynomials are linked via Cayley transformations $\mathcal{C}_{-1}$ or $\mathcal{C}_{+1}$, introduced in \cite[Def. 2.4]{mackey2006structured}. We refer the reader to \cite{mackey2006structured} for more information on these relations. In addition, in \cite[Example 3.10]{mackey2015mobius}, it was shown that the Cayley transformation is a special type of M\"obius transformation, and, by \cite[Corollary 8.6]{mackey2015mobius}, we know that M\"obius transformations preserve any strong $\ell$-ification. However, we can not ensure that they preserve companion $\ell$-ifications. Below, we provide an explanation of this. The key fact is that, after applying a M\"obius transformation, companion $\ell$-ifications can be transformed into generalized companion $\ell$-ifications which are not necessarily companion. 

Let us assume that $P(\la)$ is a $\star$-palindromic matrix polynomial of grade $k$ and $Q(\la) = \mathcal{C}_{+1}(P) := (1 - \la)^kP\left(\frac{1+\la}{1 - \la}\right)$ is a $\star$-even matrix polynomial of grade $k$. Let $\mathcal{L}_P(\la)$ be a $\star$-palindromic companion pencil of $P(\la)$. Then, the question that we aim to answer is whether there exists a $\star$-even companion pencil $\mathcal{L}_Q(\la)$ of $Q(\la)$ such that $\mathcal{L}_Q(\la) = \mathcal{C}_{+1}(\mathcal{L}_P)$. This problem is showed in the following diagram.
\begin{equation*}
\xymatrix{
	P(\la)\left(\star-{\rm palindromic}\right){\ar@{-->}[r]^{\mathcal{C}_{+1}}} {\ar@{~>}[d]_{{\rm Companion\ Pencil}}} & Q(\la)\left(\star-{\rm even}\right) {\ar@{~>}[d]^{{\rm Companion\ Pencil?}}} \\
	{\cal L}_P(\la)\left(\star-{\rm palindromic}\right){\ar@{-->}[r]^{\mathcal{C}_{+1}}} & {\cal L}_Q(\la)\left(\star-{\rm even}\right)
}
\end{equation*}
Notice that we have used the Cayley transformation $\mathcal{C}_{+1}$ and the $\star$-palindromic and $\star$-even structures, but other ones can be chosen, according to \cite[Th. 2.7]{mackey2006structured}.

The following example answers negatively the question, because we obtain a matrix pencil ${\mathcal L}_Q(\la)$, which is $\star$-even, but is not companion (${\mathcal L}_Q(\la)$ does not satisfy condition {\rm (a)} in Definition \ref{CompLification}).

\begin{Exa}\label{Counterexample}
{\rm
	Let $P(\la) = \sum_{j=0}^3 \la^j P_j \in \FF[\la]^{n \times n}$ be a cubic matrix polynomial. Then, $P(\la)$ is $\star$-palindromic if $P^\star(\la) = {\rm rev}_3P(\la)$, namely, if 
	\begin{equation*}
	P_j^\star = P_{3-j},\ {\rm for\ } j=0,1.
	\end{equation*}
	Let $Q(\la)$ be a cubic matrix polynomial defined as follows:
	\begin{equation*}
	\begin{split}
	Q(\la) :=& \mathcal{C}_{+1}(P) = (1-\la)^3\left(P_0 + \left(\dfrac{1+\la}{1-\la}\right)P_1 + \left(\dfrac{1+\la}{1-\la}\right)^2 P_2 + \left(\dfrac{1+\la}{1-\la}\right)^3 P_3 \right) \\
	=& \left(P_0 + P_1 + P_2 + P_3\right) + \la\left(-3P_0 -P_1+P_2+3P_3\right) + \la^2\left(3P_0 -P_1-P_2+3P_3\right) + \\
	&\la^3\left(-P_0+P_1-P_2+P_3\right) =: \sum_{j=0}^3 \la^j Q_j.
	\end{split}
	\end{equation*}
	By \cite[Th. 2.7]{mackey2006structured}, $Q(\la)$ is $\star$-even whenever $P(\la)$ is $\star$-palindromic.
	
	Now, let us consider the following matrix pencil 
	\begin{equation*}
	\mathcal{L}_P(\la) := \mathcal{L}_{P,0} + \la \mathcal{L}_{P,1}= \begin{bmatrix} 0 & P_0 + \la P_1 & -\la I_n \\ P_2 + \la P_3 & 0 & I_n \\ -I_n & \la I_n & 0 \end{bmatrix},
	\end{equation*}
	which is a companion pencil of $P(\la)$ (see Definition \ref{CompLification}) and, moreover, $\mathcal{L}_P(\la)$ is $\star$-palindromic whenever $P(\la)$ is, and, as above, we define ${\cal L}_Q(\la)$ as the matrix pencil
	\begin{equation*}
	\begin{split}
	\mathcal{L}_{Q}(\la) :=& \mathcal{C}_{+1}(\mathcal{L}_{P}) = (1-\la)\left(\mathcal{L}_{P,0} + \left(\dfrac{1+\la}{1-\la}\right)\mathcal{L}_{P,1}\right) = \left( \mathcal{L}_{P,0} + \mathcal{L}_{P,1} \right) + \la\left(-\mathcal{L}_{P,0} + \mathcal{L}_{P,1} \right) \\
	=& \begin{bmatrix} 0 & P_0 + P_1 + \la(-P_0 + P_1) & -I_n - \la I_n \\ P_2 + P_3 + \la (-P_2+P_3) & 0 & I_n - \la I_n \\ -I_n + \la I_n & I_n + \la I_n & 0 \end{bmatrix} \\
	=& \begin{bmatrix} 0 & \left(\frac{2Q_0 - Q_1 + 2Q_3}{4}\right) + \la \left(\frac{Q_0 - Q_2 + Q_3}{4}\right) & -I_n - \la I_n \\ \left(\frac{2Q_0 + Q_1 - Q_3}{4}\right) + \la \left(\frac{-Q_0 + Q_2 +2Q_3}{4}\right) & 0 & I_n - \la I_n \\ -I_n + \la I_n & I_n + \la I_n & 0 \end{bmatrix},
	\end{split}
	\end{equation*}
	which is $\star$-even whenever $P(\la)$ is $\star$-palindromic. However, $\mathcal{L}_Q(\la)$ is not a companion pencil of $Q(\la)$ since there are nonzero block entries of $\mathcal{L}_Q(\la)$, namely the $(1,2)$ and $(2,1)$ blocks, which do not fulfill condition {\rm (a)} in Definition \ref{CompLification}.
}
\end{Exa}  

As a consequence of Example \ref{Counterexample}, M\"obius transformations do not preserve the property of being a companion $\ell$-ification. For this reason, the proof of Theorem \ref{NoStructCompLific} has been carried out for all the structures in Definition \ref{Structures-Mobius} separately.

The question on whether there are structured companion $\ell$-ifications of structured matrix polynomials of grade $k = (2d) \ell$ is still open in general. We have only proved, in Theorem \ref{NoStructCompLific}, the case $d=1$ and $\ell=2$, that is, there are no structured companion quadratifications of structured quartic matrix polynomials.

\section{Conclusions}\label{conclusion_sec}

In this paper we have presented, for the first time, structured (generalized) companion $\ell$-ifications of structured matrix polynomials of grade $k=(2d+1)\ell$, given in the monomial basis, for the main structures of matrix polynomials that are frequently considered in the literature (namely, (skew)-symmetric, (skew-)Hermitian, (anti-)palindromic, and alternating structures). The constructions are structured versions of the block-Kronecker $\ell$-ifications presented in \cite{dopico2019block}. We have also determined the smallest number of nonzero block entries in these constructions, and we have provided a procedure to construct sparse structured $\ell$-ifications within this family. Finally, we have shown that there are no structured companion quadratifications for quartic matrix polynomials. 

As a natural continuation of this work, some lines of future research are the following: (1) To show whether or not there are structured companion $\ell$-ifications for matrix polynomials of grade $k=(2d)\ell$; (2) to look for (not necessarily structured) $\ell$-ifications for matrix polynomials given in other bases than the monomial basis.

%\bigskip
%{\bf Acknowledgment.} \\ \\ \\

%%%%%%%%%%%%%%%%%%%%%%%%%%%%%%%%%%%%%%%%%%%%%%%%%%%%%%%%%%%%%
\bibliographystyle{abbrv}
\bibliography{Bibliography}
%\begin{thebibliography}{1}
%\bibitem{schneider-84}
%Hans Schneider.
%\newblock  Theorems on M-splittings of a singular M-matrix which
%           depend on graph structure.
%\newblock  {\em Linear Algebra and its Applications}, 58:407--424, 1984.
%
%\bibitem{varga-62}
%Richard S.~Varga.
%\newblock  {\em Matrix Iterative Analysis}.
%\newblock  Prentice-Hall, Englewood Cliffs, New Jersey, 1962.
%
%\bibitem{ela-paper}
%S.~Friedland and H.~Schneider. 
%\newblock  Spectra of expansion graphs. 
%\newblock  {\em Electronic Journal of Linear Algebra}, 6:2--10, 1999.
%
%\end{thebibliography}
\end{document}